\documentclass[smallextended]{svjour3}

\smartqed
\usepackage{graphicx}
\usepackage{amsmath}
\usepackage{amsfonts}
\usepackage{amssymb}
\usepackage{overpic}

\usepackage[pdftex,%
            colorlinks,linkcolor=red,citecolor=blue,urlcolor=blue,%
            pdfstartview=FitH,%
            pdfauthor={Mikhail Skopenkov, Helmut Pottmann, Philipp Grohs},%
            pdftitle={Ruled Laguerre Minimal Surfaces},%
            pdfkeywords={Laguerre geometry, Laguerre minimal ruled surface,%
            biharmonic function.}%
           ]{hyperref}
\usepackage{hypernat}
\usepackage{anysize}
\usepackage{url}

\usepackage{setspace}
\usepackage{longtable}

\usepackage{pgf,tikz}
\usetikzlibrary{arrows}

\numberwithin{equation}{section}

\journalname{\mscomm{This is an updated version of the paper published in Math. Z. 272:1--2 (2012) 645--674. \hspace{4cm}}} 

\newcommand{\mscomm}[1]{#1}

\begin{document}

\title{Ruled Laguerre minimal surfaces
}



\author{Mikhail Skopenkov         \and
        Helmut Pottmann           \and
        Philipp Grohs
}

\authorrunning{M. Skopenkov, H. Pottmann, Ph. Grohs} 

\institute{M. Skopenkov \at
              Institute for information transmission problems of the Russian Academy of Sciences,
              Bolshoy Karetny per. 19, bld. 1, Moscow, 127994, Russian Federation \\
              \email{skopenkov@rambler.ru}             \\
              \emph{Present address:} King Abdullah University of Science and Technology,
              P.O. Box 2187, 4700 Thuwal, 23955-6900, Kingdom of Saudi Arabia
           \and
           H. Pottmann \at
              King Abdullah University of Science and Technology, 4700 Thuwal, 23955-6900, Kingdom of Saudi Arabia
              \\
              \email{helmut.pottmann@kaust.edu.sa}
           \and
           Ph. Grohs \at
              Seminar for Applied Mathematics, ETH Zentrum, R\"amistrasse 101, 8092 Zurich, Switzerland\\
              \email{pgrohs@sam.math.ethz.ch}
}

\date{}

\maketitle

\begin{abstract}
A \emph{Laguerre minimal surface} is an immersed surface in $\mathbb{R}^3$ being an
extremal of the functional $\int (H^2/K-1)dA$. In the present paper, we prove that
the only {\em ruled} Laguerre minimal surfaces are up to isometry
the surfaces $\mathbf{R}(\varphi,\lambda) = ( A\varphi,\, B\varphi,\, C\varphi +D\cos 2\varphi\, )
+ \lambda\left(\sin \varphi,\, \cos \varphi ,\, 0\,\right)$,
where $A,B,C,D\in\mathbb{R}$ are fixed.
To achieve invariance under Laguerre transformations, we also derive
all Laguerre minimal surfaces that are enveloped by a family of cones.
The methodology is based on the isotropic model of Laguerre geometry.
In this model a Laguerre minimal surface enveloped by a family of cones corresponds to
a graph of a biharmonic function carrying a family of isotropic circles.
We classify such functions by showing that the top view of the family of
circles is a pencil.

\end{abstract}

\tableofcontents

\begin{figure}[htbp]
\centering
\includegraphics[width=0.9\textwidth]{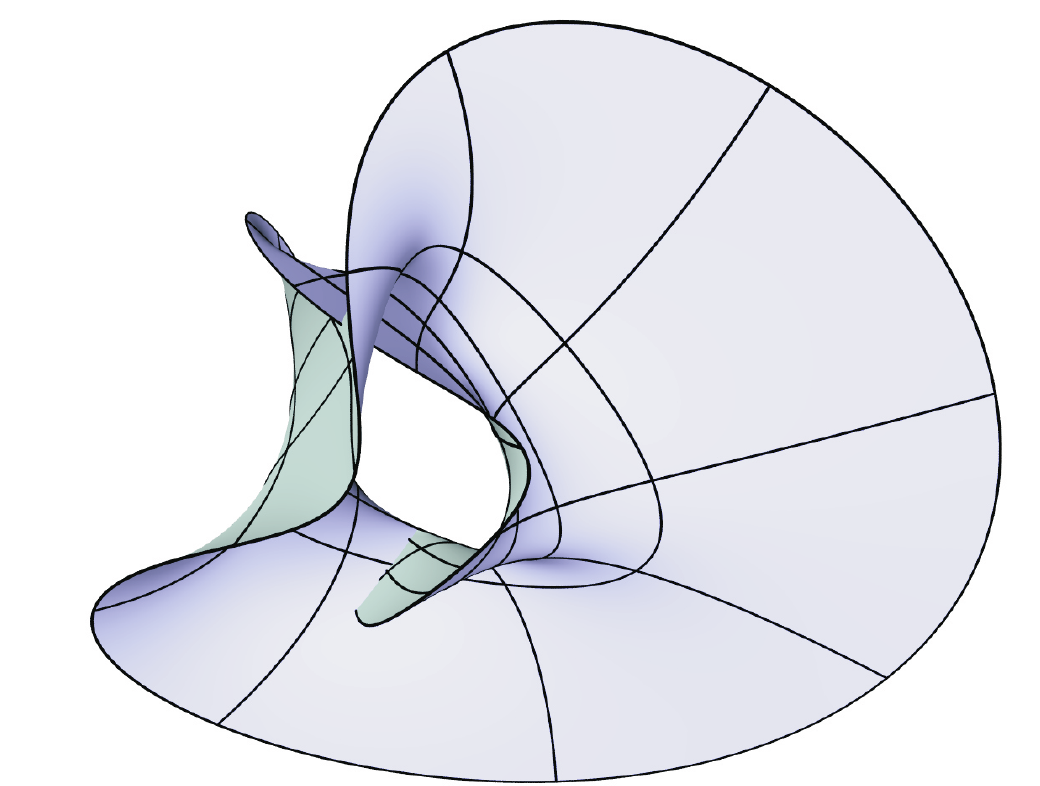}
 \caption{A general L-minimal surface enveloped by a hyperbolic family of cones. For details refer to Definition~\ref{def-hyperbolic} and
 Corollary~\ref{classification-hyperbolic}.}
 \label{figure-hyperbolic-general}
\end{figure}

\section{Introduction}\label{intro}

This is the third in a series of papers \cite{PGM:2009,PGB:2010} where
we develop and study a novel approach to the Laguerre differential geometry of
immersed Legendre surfaces in $\mathbb{R}^3$.
Laguerre geometry is the Euclidean geometry of oriented planes and spheres. Besides M\"obius and Lie geometry, it
is one member of the three classical sphere geometries in $\mathbb{R}^3$ \cite{cecil:1992}.

After the seminal work \cite{blaschke:1929} of Blaschke on this topic in the beginning of the 20th century, this classical topic has again found the interest of differential geometers.

For instance, the celebrated work on discrete differential geometry by Bobenko
and coworkers \cite{bobenko-2005-ddg,bobenko:org,bobenko-2006-ms} heavily uses this
theory in developing discrete counterparts to continuous definitions.

On the practical side, recent research in architectural geometry identified
certain classes of polyhedral surfaces,
namely conical meshes \cite{liu-2006-cm,pottmann-tr163} and meshes
with edge offsets \cite{pottmannetal:2007}, as particularly suitable
for the representation and fabrication of architectural freeform structures. These types of
polyhedral surfaces are actually objects of
Laguerre sphere geometry \cite{bobenko:org,liu-2006-cm,pottmannetal:2007,pottmannliu:2007,flex-min,PGB:2010}.

The aim is to study (discrete, see \cite{PGB:2010}, and continuous, see \cite{PGM:2009})
minimizers of geometric energies which are invariant under Laguerre
transformations. The simplest energy of this type has been
introduced by Blaschke
\cite{blaschke:1924,blaschke:1925,blaschke:1929}. Using mean
curvature $H$, Gaussian curvature $K$, and the surface area element
$dA$ of an immersion $\mathbf{r}:\mathbb{R}^2\to \mathbb{R}^3$, it can be expressed as the surface integral
\begin{equation}
 \Omega=\int ({H^2-K) / K}\, dA. \label{Lfunctional}
\end{equation}
Though the quantities $H,K,A$ used for the definition are not
objects of Laguerre geometry, the functional $\Omega$
invariant under Laguerre transformations.
An immersion $\mathbf{r}:\mathbb{R}^2\to \mathbb{R}^3$ with $K\ne0$, which is an extremal of the energy $\Omega$ with respect to compactly supported variations, is called a {\em Laguerre-minimal \textup{(}L-minimal\textup{)} surface}.
A Euclidean minimal surface (\mscomm{outside planar points}) is a particular case of such a surface.

In 1842 
Catalan proved that the only ruled Euclidean minimal surfaces are the plane and the helicoid. A surface is \emph{ruled}, if each point of the surface belongs to a line segment contained in the surface. One of the main purposes of this paper is to describe all ruled Laguerre minimal surfaces.

The property of a surface to be ruled is not invariant under Laguerre transformations.
A line in a surface may be taken to a cone or cylinder of revolution touching the image of the surface along a curve.
Hence, we will also derive all Laguerre minimal surfaces which are \emph{enveloped by a family of cones}; see Figure~\ref{figure-hyperbolic-general}.
In the following, when speaking of a \emph{cone}, we will always assume this to be a cone of revolution, including the special
cases of a rotational cylinder and a line.
\mscomm{We say that a surface is \mscomm{\emph{enveloped by a $1$-parametric family $\mathcal{F}$ of cones},} if through each point of the surface there passes a cone from $\mathcal{F}$ which is tangent to the surface along a continuous curve (containing the point).}

Our approach is based on a recent result \cite{PGM:2009} which shows that Laguerre minimal surfaces
appear as graphs of biharmonic functions in the isotropic model of Laguerre geometry.
This result has various corollaries on Laguerre minimal surfaces, geometric optics and linear elasticity.

A Laguerre minimal surface
enveloped by a family of cones corresponds
to a graph of a biharmonic function carrying a family
of isotropic circles. We classify such
functions. In particular we show that besides a few exceptions the top view of such a family of isotropic circles must be a pencil.
In the course of the proof of this result we also develop a new
symmetry principle for biharmonic functions.

\subsection{Previous work}

\noindent Differential geometry
in the three classical sphere geometries of M\"o\-bi\-us, Laguerre and
Lie, respectively, is the subject of Blaschke's third volume on
differential geometry \cite{blaschke:1929}. For a more modern
treatment we refer to Cecil \cite{cecil:1992}.
Here we focus on contributions to L-minimal surfaces.
Many L-minimal surfaces are found in the work of Blaschke
\cite{blaschke:1924,blaschke:1925,blaschke:1929} and in papers by
his student K\"onig \cite{koenig:1926,koenig:1928}.

Recently, this topic found again the interest of differential geometers, after Pinkall used Lie sphere geometry
to classify Dupin hypersurfaces in $\mathbb{R}^4$ \cite{pinkall:1985}.
The stability of L-minimal surfaces has been
analyzed by Palmer \cite{palmer:1999}; he also showed that these surfaces are indeed local minimizers of (\ref{Lfunctional}).
Li and Wang studied L-minimal surfaces using
the Laguerre Gauss map and a Weierstrass representation \cite{springerlink:10.1007/s10114-005-0642-1,springerlink:10.1007/s00025-008-0314-4}, cf.~also \cite{dorfmeister:2007,gollek}. Musso and Nicolodi studied L-minimal surfaces by the method of moving frames
\cite{musso:1996}. L-minimal surfaces which are envelopes of a family of cones include as special cases the L-minimal canal surfaces described by Musso and Nicolodi 
\cite{musso:1995}.
The description of all \emph{ruled} L-minimal surfaces is not directly accessible by the methods known before.

\subsection{Contributions}

Our main result is a description of all the
L-minimal surfaces which are envelopes of an analytic family $\mathcal F$ of cones of revolution.
We show that for any such surface (besides a sphere and a parabolic cyclide) the family $\mathcal F$
belongs to one of three simple types;
see Definitions~\ref{def-elliptic}, \ref{def-hyperbolic}, \ref{def-parabolic} and Corollary~\ref{cor-pencil}.
For each type we represent the surface as a convolution of certain basic surfaces; see Examples~\ref{ex1}--\ref{ex11} and Corollaries~\ref{classification-elliptic},~ \ref{classification-hyperbolic},~\ref{classification-parabolic}.

As an application we show the following:

\begin{theorem}\label{th-ruled-intro}\label{th-kineruled}
	A ruled Laguerre minimal surface is up to isometry a piece of the surface
	\begin{equation}\label{eq-ruled-intro}
		\mathbf{R}(\varphi,\lambda) =
                \left( A\varphi,\, B\varphi,\, C\varphi +D\cos 2\varphi\, \right)
		+ \lambda\left(\sin \varphi,\, \cos \varphi ,\, 0\,\right),
	\end{equation}
for some $A,B,C,D\in\mathbb{R}$ such that $C^2+D^2\ne0$.
\end{theorem}


In other words, a ruled L-minimal surface can be constructed
as a superposition of a frequency $1$ rotating motion
of a line in a plane, a frequency $2$ ``harmonic oscillation'' between two lines
parallel to the plane, and a constant-speed translation. Equivalently, a ruled L-minimal surface is a convolution of a helicoid ($A=B=D=0$), the Pl\"ucker conoid ($A=B=C=0$), and a cycloid (limit case $C,D\to 0$); see Examples~\ref{ex1}--\ref{ex3} and Theorem~\ref{th-ruled} for accurate statement.

One more result of the paper is a description of all the 
i-Willmore surfaces carrying an analytic family of i-circles; see Table~\ref{isotropic} for definitions, and Corollary~\ref{th-pencil}, Theorems~\ref{i-classification-elliptic}, \ref{i-classification-hyperbolic}, \ref{i-classification-parabolic} for the statements.


\subsection{Organization of the paper}

In \S\ref{sec-isotropic} we give an introduction to isotropic and Laguerre geometries and translate
the investigated problem to the language of isotropic geometry. This section does not contain new results. In \S\ref{pencil} we state and prove the Pencil Theorem~\ref{th1}, which describes the possible families $\mathcal{F}$ of cones. In \S\ref{classification} we describe the Laguerre minimal surfaces for each type of cone family $\mathcal{F}$ and prove Theorem~\ref{th-ruled-intro}.

\section{Isotropic model of Laguerre geometry}\label{sec-isotropic}

\subsection{Isotropic geometry}\label{ssec:isotropic}

 {\em Isotropic geometry} has been systematically developed
by Strubecker
\cite{strubecker:1941,strubecker:1942a,strubecker:1942} in the
1940s; a good overview of the many results is provided in
the monograph by Sachs \cite{Sachs:1990}.

The \emph{isotropic space} is the affine space $\mathbb{R}^3$ 
equipped with the norm $\|(x,y,z)\|_i:=\sqrt{x^2+y^2}$.
The invariants of affine transformations preserving this norm are subject of \emph{isotropic geometry}.
\noindent

The projection $(x,y,z) \mapsto (x,y,0)$
of isotropic space onto the $xy$-plane is called \emph{top view}.
Basic objects of isotropic geometry and their definitions (from the point of view of Euclidean geometry in isotropic space) are given in the first two columns of Table~\ref{isotropic}; see also Figure~\ref{fig:isotropic_illustration}. We return to the third column of the table further.

\begin{table}[htbp]
\caption{Basic objects of isotropic geometry as images of surfaces in the isotropic model of Laguerre geometry.}
\label{isotropic}
\begin{tabular}{|p{4.0cm}|p{7.0cm}|p{3.5cm}|}
\hline
Object\newline of isotropic geometry & Definition                 & Corresponding surface\newline
in Laguerre geometry \\
\hline
point                        & point in isotropic space           & oriented plane \\
non-isotropic line           & line nonparallel to the $z$-axis  & cone \\
non-isotropic plane          & plane nonparallel to the $z$-axis & oriented sphere
\\
i-circle of elliptic type    & ellipse whose top view is a circle  & cone \\
i-circle of parabolic type   & parabola with $z$-parallel axis     & cone \\
i-sphere of parabolic type   & paraboloid of revolution with \ $z$-parallel axis & oriented sphere \\
i-paraboloid                 & graph of a quadratic function \ $z=F(x,y)$ & parabolic cyclide
\newline \emph{or} oriented sphere \\
i-Willmore surface          & graph of a (multi-valued) biharmonic function $z=F(x,y)$ & L-minimal surface \\
\hline
\end{tabular}
\end{table}

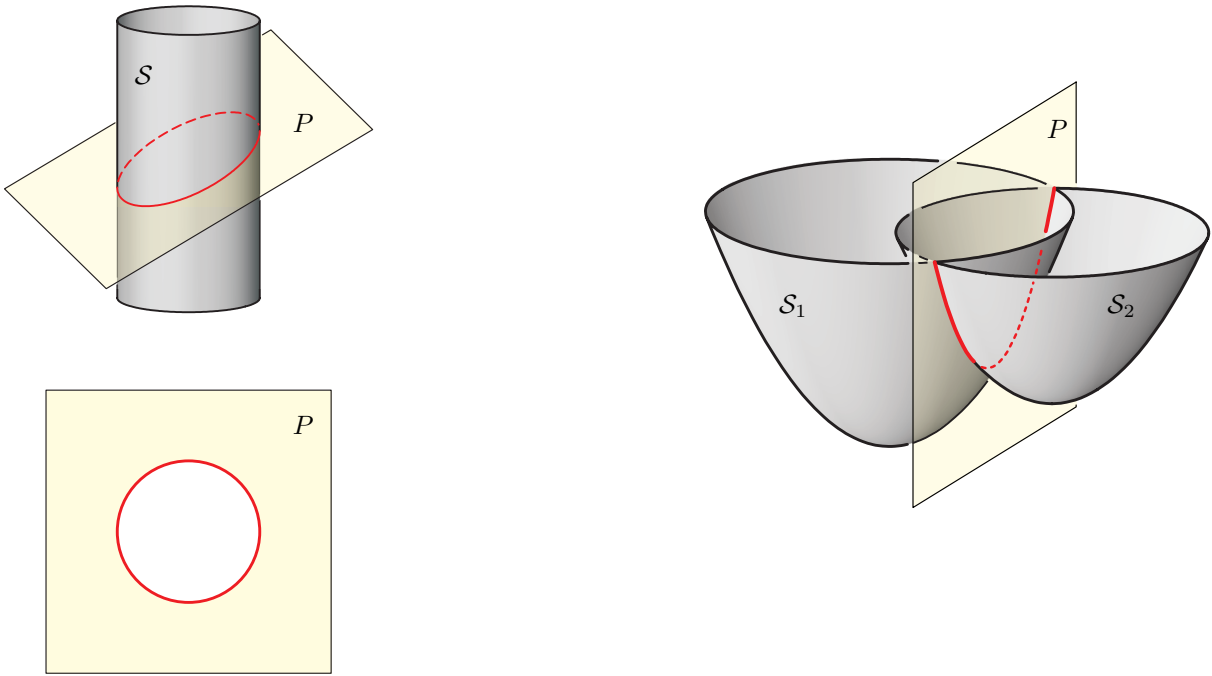
\begin{figure}[htbp]
\centering
\begin{overpic}[width=\textwidth]{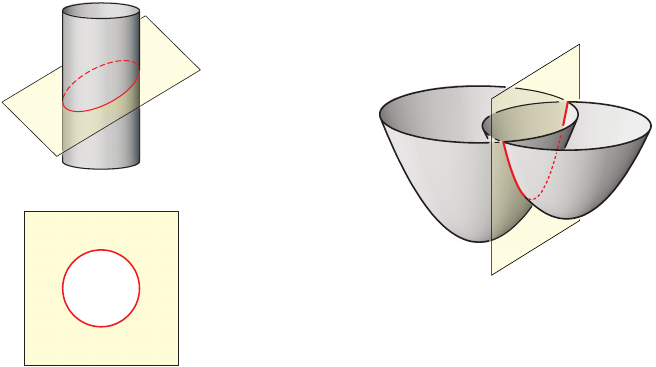}
\put(24,45){$P$} \put(24,20){$P$} \put(11,49){$\mathcal{S}$}
\put(86,44.5){$P$} \put(64,30){$\mathcal{S}_1$}
\put(91,30){$\mathcal{S}_2$}
\end{overpic}
 \caption{(Left) An i-circle of elliptic type
 is the intersection curve of a vertical round cylinder $\mathcal{S}$
 and a non-isotropic plane $P$. When viewed from the top, the i-circle
 is a Euclidean circle. (Right) An i-circle of parabolic type is a parabola with $z$-parallel
 axis. This curve appears as the intersection curve of two i-spheres, $\mathcal{S}_1$ and $\mathcal{S}_2$,
 of parabolic type with the same i-mean curvature.
 For more details refer to Table~\ref{isotropic} and \S\ref{ssec:isotropic}.}
 \label{fig:isotropic_illustration}
\end{figure}

In isotropic space there exists a counterpart to M\"obius geometry.
One puts i-spheres of
parabolic type and non-isotropic planes into the same class 
of {\em isotropic M\"obius spheres} (\emph{i-M-spheres}); they are given by
the equation $z=\frac{a}{2}(x^2+y^2)+bx+cy+d$ for some $a,b,c,d\in\mathbb{R}$.
The coefficient $a$ in this representation is called the \emph{i-mean curvature} of the i-M-sphere.
An intersection curve of two i-M-spheres is called an {\em i-M-circle}; it may be an i-circle of elliptic or
parabolic type or a non-isotropic straight line.


Similarly to Euclidean M\"obius geometry, where an ideal point is added to $\mathbb{R}^3$, in isotropic M\"obius geometry an \emph{ideal line} $\ell_\infty$ is added to $\mathbb{R}^3$. The resulting space $\mathbb{R}^3\cup\ell_\infty$ is called \emph{extended isotropic space}.
By definition, an i-M-sphere with i-mean curvature $a$ intersects the ideal line at the point $a\in \ell_\infty$.

A map acting on $\mathbb{R}^3\cup \ell_\infty$ is called an {\em isotropic M\"obius \textup{(}i-M\textup{)} transformation}, if it takes i-M-spheres to i-M-spheres (and hence i-M-circles to i-M-circles). The top view of an i-M-transformation is a planar Euclidean M\"obius transformation. Basic i-M-transformations which together with the translation $(x,y,z)\mapsto (x+1,y,z)$ generate the whole group of i-M-transformations are given in the first column of Table~\ref{isotropic-trans}. Here $R^{\theta}$ is the counterclockwise rotation through an angle $\theta$ around the $z$-axis.



\begin{table}[htbp]
\caption{Basic isotropic M\"obius transformations as images of Laguerre transformations in the isotropic model of Laguerre geometry.}
\label{isotropic-trans}
\begin{tabular}{|l|l|}
\hline
i-M-transformation                  & Corresponding L-transformation \\
\hline
$(x,y,z)\mapsto R^\theta (x,y,z)$  & rotation $R^\theta$ \\
$(x,y,z)\mapsto (x,y,z+ax+by)$     & translation by vector $(a,b,0)$ \\
$(x,y,z)\mapsto (x,y,z+x^2+y^2-1)$ & translation  by vector $(0,0,1)$ \\
$(x,y,z)\mapsto (x,y,z+h)$         & $h$-offset operation \\
$(x,y,z)\mapsto (x,y,az)$          & homothety with coefficient $a$ \\
$(x,y,z)\mapsto (x,y,z)/(x^2+y^2)$ & reflection with respect to the plane $z=0$ \\
$(x,y,z)\mapsto (x,y,z)/\sqrt{2}$  & transformation $\Lambda$\\
\hline
\end{tabular}
\end{table}

\subsection{Laguerre Geometry} \label{ssec:laguerre}

A \emph{contact element} is a pair $(r,P)$,  where $r$ is a point in $\mathbb{R}^3$,
and $P$ is an oriented plane passing through the point $r$.
Denote by $ST\mathbb{R}^3$ 
the space of all contact elements.

To an immersed oriented surface $\Phi$ in $\mathbb{R}^3$ assign the set of all the contact elements $(r,P)\in ST\mathbb{R}^3$ such that  
$r\in\Phi$ and $P$ is the oriented tangent plane to $\Phi$ at the point $r$.
We get a \emph{Legendre surface}, i.~e., the image of an immersion $(\mathbf{r},\mathbf{P}):\mathbb{R}^2\to ST\mathbb{R}^3$ such that $d\mathbf{r}(u,v)\parallel\mathbf{P}(u,v)$.
Further we do not distinguish between an immersed surface in $\mathbb{R}^3$ and the corresponding Legendre surface, if no confusion arises.

An example of a Legendre surface not obtained from an immersed one is a \emph{point}, or a \emph{sphere of radius $0$}, which is the set of all the contact elements $(r,P)$ such that $r=r_0$ is fixed and $P\ni r_0$ is arbitrary.

A \emph{Laguerre transformation \textup{(}L-transformation\textup{)}} is a bijective map $ST\mathbb{R}^3\to ST\mathbb{R}^3$ taking oriented planes to oriented planes and oriented spheres (possibly of radius $0$) to oriented spheres (possibly of radius $0$).
The invariants of Laguerre transformations are the subject of \emph{Laguerre geometry}
\cite{blaschke:1929,cecil:1992}.

Note that an L-transformation does not in general preserve points,
since those are seen as spheres of radius $0$ and may be mapped to other
spheres. A simple example of an L-transformation is the \emph{$h$-offset operation},
translating a contact element $(r,P)$ by the vector $h\mathbf{n}$, where $\mathbf{n}$
is the positive unit normal vector to the oriented plane $P$.

A Laguerre transformation is uniquely defined by its action on the set of oriented planes. A \emph{Hesse normal form} of an oriented plane $P$ is the equation $n_1x+n_2y+n_3z+h=0$ of the plane such that $(n_1,n_2,n_3)$ is the positive unit normal vector to the oriented plane.

Consider the Laguerre transformation $\Lambda$ taking an oriented plane in the Hesse normal form $n_1x+n_2y+n_3z+h=0$ to the oriented plane $n_1x+n_2y+\frac{1}{2}(3n_3+1)z+h=0$ with obvious orientation. Denote by $\tilde{\mathbf{r}}(u,v)$ the surface obtained from a surface $\mathbf{r}(u,v)$ by the transformation and parametrized so that the tangent planes to the surfaces $\tilde{\mathbf{r}}(u,v)$ and $\mathbf{r}(u,v)$ are parallel at points having the same parameters $u$ and $v$. This notation is convenient in our classification results which follow.

Similarly, denote by ${\mathbf{r}}^\theta(u,v)$ the surface obtained from a surface $\mathbf{r}(u,v)$ by the rotation $R^\theta$ and parametrized so that the tangent planes to the surfaces ${\mathbf{r}}^\theta(u,v)$ and $\mathbf{r}(u,v)$ are parallel for each $(u,v)\in\mathbb{R}^2$.

More examples of Laguerre transformations are given in the second column of Table~\ref{isotropic-trans}.

For a pair of parallel 
oriented planes $P_1$ and $P_2$ in Hesse normal forms $n_1x+n_2y+n_3z+h_1=0$ and $n_1x+n_2y+n_3z+h_2=0$, denote by $a_1P_1\oplus a_2P_2$ the plane in Hesse normal form $n_1x+n_2y+n_3z+a_1h_1+a_2h_2=0$.
Define a \emph{convolution surface} $a_1\Phi_1\oplus a_2\Phi_2$
of two Legendre surfaces $\Phi_1$ and $\Phi_2$ to be the Legendre surface formed by the contact elements of the form $(a_1r_1+a_2r_2,a_1P_1\oplus a_2P_2)$, where
$(r_1,P_1)\in\Phi_1$ and $(r_2,P_2)\in\Phi_2$ run through all the pairs of contact elements with $P_1$ parallel to~$P_2$.
In particular, if the surfaces $\Phi_1$ and $\Phi_2$ enclose convex bodies then the convolution surface $\Phi_1\oplus\Phi_2$ encloses the Minkowsky sum of the bodies.

\subsection{Isotropic model of Laguerre geometry}
To each oriented plane in the Hesse normal form $n_1x+n_2y+n_3z+h=0$,
with 
$n_3\ne-1$ assign the point
\begin{equation} \frac{1}{n_3+1}(n_1,n_2,h) \label{Pi} \end{equation}
of the isotropic space.
To an oriented plane in the Hesse normal form $-z+h=0$ assign
the ideal point $h\in\ell_\infty$.
This induces a map from 
the space $ST\mathbb{R}^3$ to the extended isotropic space $\mathbb{R}^3\cup\ell_\infty$.
The map provides the {\em isotropic model of Laguerre geometry}.
For a more geometric definition see~\cite{pottmann-1998-alcagd}.

\mscomm{To map an oriented (Legendre) surface $\Phi$ to the isotropic model, we consider the set $\Phi^i$ of points $P^i$, where $P$ runs through all oriented tangent planes to $\Phi$.}


As examples, consider the pairs of surfaces $\Phi^i$ and $\Phi$ given in the last two columns of Table~\ref{isotropic}.

An oriented sphere with center $(m_1,m_2,m_3)$, radius $R$, \mscomm{and inwards oriented normals} is mapped to the isotropic M\"obius sphere
\begin{equation}
  z={R+m_3 \over 2}(x^2+y^2)-m_1x-m_2y+{R-m_3 \over 2} \label{iso-sphere}.
\end{equation}
%

A cone viewed as the common tangent planes of two oriented spheres is mapped
to the common points of two i-M-spheres (= i-M-circle) in the isotropic model.
In particular, a line is mapped to an i-M-circle of the form
\begin{equation}\label{line}
\begin{cases}
z=m_3(x^2+y^2-1)-m_1x-m_2y,\\
z=n_3(x^2+y^2-1)-n_1x-n_2y.
\end{cases}
\end{equation}

\mscomm{The top view of $\Phi^i$ is actually the stereographic projection of the Gaussian spherical image of $\Phi$ from the point $(0,0,-1)$ to the $xy$-plane. In particular, if the Gaussian curvature $K\ne 0$ everywhere, then $\Phi^i$ is locally a graph of a function.}

\mscomm{Notice that the definition of Laguerre minimal surfaces is local and requires $K\ne 0$. Restricting to a part of the surface (if needed), in what follows we assume the following.}

\smallskip \noindent
\mscomm{{\bf Condition (*)} $\Phi$ is an oriented \emph{embedded} surface (i.e. the image of a proper injective $C^2$ map of an open disk to $\mathbb{R}^3$ with nondegenerate differential at each point, equipped with an oriented unit normal continuously depending on the point) with \emph{nowhere vanishing Gaussian curvature} $K$, and $\Phi^i$ is the \emph{graph of a $C^1$ function} $F$ defined in a planar domain.}
\smallskip

\mscomm{We are going to see soon that Laguerre minimal surfaces are analytic. An analytic continuation of $\Phi$ may easily acquire points with $K=0$, self-intersections, singularities, and $F$ may become multi-valued. The resulting complete surface gives a lot of geometric insight but needs not satisfy condition~(*) globally.}

The surface $\Phi$ can be reconstructed given the surface $\Phi^i$:

\begin{proposition}\label{cor1} \textup{(cf. \cite[Corollary~2]{PGM:2009})}
\mscomm{Assume (*).} Then the surface $\Phi$ can be parametrized as follows:
\begin{equation}\label{eq-cor1}
\mathbf{r}(x,y)=
\frac{1}{x^2+y^2+1}\left(
\text{\begin{tabular}{c}
$(x{}^2-y{}^2-1)F_x+2x y F_y-2x F$\\
$(y{}^2-x{}^2-1)F_y+2x y F_x-2y F$\\
$2x F_x+2y F_y-2 F$
\end{tabular}}
\right).
\end{equation}
\end{proposition}

\begin{proof}\mscomm{
Let $(n_1,n_2,n_3)$ be the oriented unit normal at a point $(r_1,r_2,r_3)$ of $\Phi$. The oriented tangent plane $P$ at the point is given by $n_1x+n_2y+n_3z-n_1r_1-n_2r_2-n_3r_3=0$. By the definition of $P^i$ we get
\begin{equation}
  F\left(\frac{n_1}{n_3+1},\frac{n_2}{n_3+1}\right) =
  -\frac{n_1 r_1+n_2 r_2+n_3 r_3}{n_3+1}.\label{eq-f}
\end{equation}
Now let $(x,y)=\left(\tfrac{n_1}{n_3+1},\tfrac{n_2}{n_3+1}\right)$ be the stereographic projection of $(n_1,n_2,n_3)$ from $(0,0,-1)$. Then $n_1^2+n_2^2+n_3^2=1$ implies that $\tfrac{n_3}{n_3+1}=\tfrac{1}{2}(1-x^2-y^2)$.
Substituting these expressions into~\eqref{eq-f}, we get
\begin{equation}
  F\left(x,y\right) =
  -x r_1-y r_2-\tfrac{1}{2}(1-x^2-y^2) r_3.\label{eq-f1}
\end{equation}
Consider $(r_1,r_2,r_3)$ as functions in $(x,y)$.
Differentiating \eqref{eq-f1} with respect to $x$,
and using the condition $
n_1\tfrac{\partial}{\partial x}r_1+n_2\tfrac{\partial}{\partial x}r_2+n_3\tfrac{\partial}{\partial x}r_3=0$ that $(n_1,n_2,n_3)$ is normal to~$\Phi$, we get
\begin{equation}
  F_x\left(x,y\right) =
  x r_3 -r_1.\label{eq-fx}
\end{equation}
Analogously we get
\begin{equation}
  F_y\left(x,y\right) =
  y r_3 -r_2.\label{eq-fy}
\end{equation}
Solving ~\eqref{eq-f1}--\eqref{eq-fy} as a linear system in $r_1,r_2,r_3$ we get~\eqref{eq-cor1}.
}
\end{proof}

\mscomm{Now we make} the following key observation; see Figure~\ref{figure-cones}:

\begin{proposition} \label{observation} \mscomm{Assume (*).}
The surface $\Phi$ is enveloped by a family of cones if and only if \mscomm{through each point of the surface $\Phi^i$ there passes an arc of an i-M-circle fully contained in $\Phi^i$}.
\end{proposition}

\begin{proof}\mscomm{If a cone $C$ is tangent to $\Phi$ along a continuous curve, then the tangent plane is not constant along the curve (otherwise $C$ and $\Phi$ are tangent along a line segment which implies vanishing Gaussian curvature), thus an arc of the i-M-circle $C^i$ is fully contained in $\Phi^i$, and vice versa. This argument is valid as well when $C$ degenerates to a line segment on $\Phi$.}
\end{proof}

\begin{theorem} \label{pgm} \textup{\cite[Theorem~1]{PGM:2009}}
\mscomm{Assume (*).} The surface $\Phi$ is Laguerre minimal if and only if \mscomm{$F$
is biharmonic}, i.~e., satisfies the equation $\Delta(\Delta(F))=0$.
\end{theorem}

\mscomm{The graph of a biharmonic function, as well as of its analytic continuation, which may become multi-valued, is called an \emph{i-Willmore surface}. One can always avoid multi-valued functions by restriction to a smaller part of the surface.}

Convolution surface $a_1\Phi_1\oplus a_2\Phi_2$ corresponds in the isotropic model to the linear combination of the two multi-valued functions whose graphs are $\Phi_1^i$ and $\Phi_2^i$. Thus a convolution surface of two L-minimal surfaces is L-minimal~\cite[Corollary~3]{PGM:2009}.

L-transformations correspond to i-M-transformations in the isotropic model and vice versa. Some examples are given in Table~\ref{isotropic-trans}.
Invariance of L-minimal surfaces under L-transformations is translated in the isotropic model as follows:

\begin{theorem} \label{cl3} \textup{\cite[Theorem~1]{PGM:2009}} Suppose that $F$ is a graph of a function biharmonic in a region
$U\subset\mathbb{R}^2$ and
$m:\mathbb{R}^3\cup\ell_\infty\to\mathbb{R}^3\cup\ell_\infty$
is an isotropic M\"obius transformation.
Then $m(F)$ is a graph of a function biharmonic in the top view of $m(U\times\mathbb{R})-\ell_\infty$.
\end{theorem}


\begin{figure}[htbp]
\centering
\includegraphics[width=\textwidth]{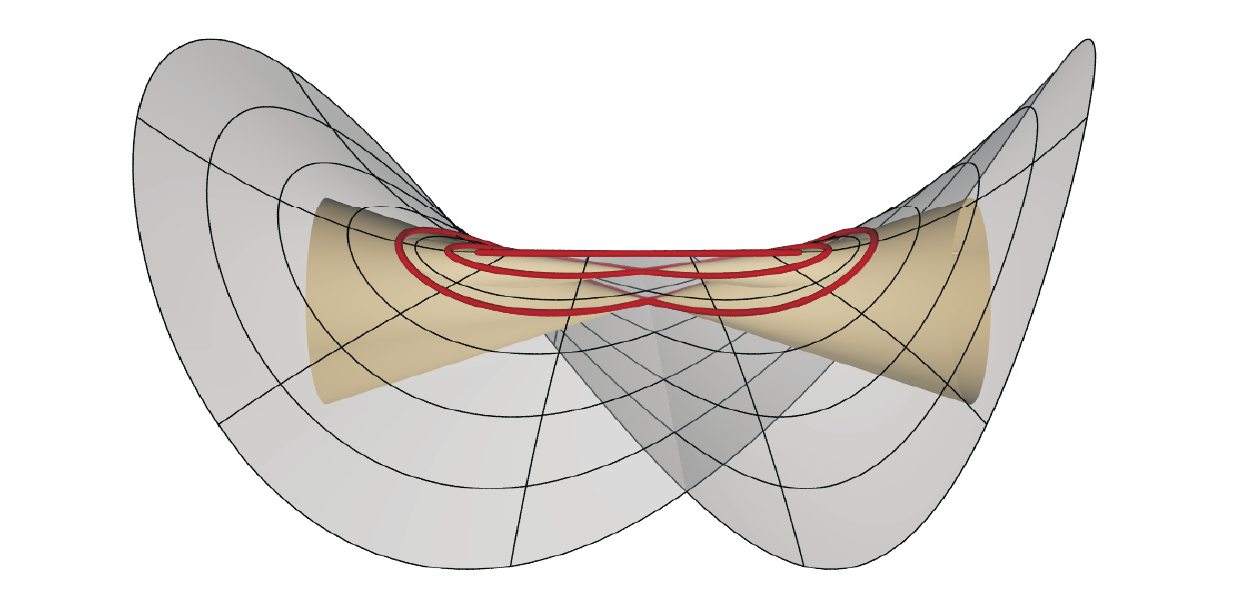}
 \caption{The L-minimal surface $\mathbf{r}_5$ arising as the envelope of a family of cones; see Example~\ref{ex5} for the details.}
 \label{figure-cones}
\end{figure}

\bigskip
\noindent\emph{Plan of the proof of Theorem 1}
To a ruled L-minimal surface there corresponds an i-Willmore surface containing a family of i-M-circles
in the isotropic model.

First we show that the top view of the family of i-M-circles is a pencil. In other words,
all the rulings of the L-minimal surface are parallel to one plane.

Then by appropriate choice of coordinates we transform the pencil into a pencil of lines.
In the latter case we describe all possible i-Willmore surfaces by solving the biharmonic
equation explicitly.

Returning to the Euclidean model we get a description of all ruled L-minimal surfaces.

\section{Biharmonic functions carrying a family of i-circles}\label{pencil}


\subsection{Statement of the Pencil theorem}

In this section we show that the top view of a family of i-M circles contained in a graph of a biharmonic function (besides a few exceptions) is a \emph{pencil},
i. e., a set of circles orthogonal to two fixed ones. This implies that the rulings of a ruled L-minimal surface are parallel to one plane.
Denote by $I=[0;1]$.

\begin{theorem}[Pencil theorem] \label{th1} Let $F(x,y)$ be a biharmonic function in a region $U\subset\mathbb{R}^2$. Let $S_t$, $t\in I$, be an analytic family of circles in the plane. Suppose that for each $t\in I$  we have $S_t\cap U\ne\emptyset$ and the restriction $F\left|_{S_t\cap U}\right.$ is a restriction of a linear function. Then either $S_t$, $t\in I$, is a pencil of circles or
\begin{equation}\label{eq-th1}
F(x,y)=A((x-a)^2+(y-b)^2)+\frac{B(x-c)^2+C(x-c)(y-d)+D(y-d)^2}{(x-c)^2+(y-d)^2}
\end{equation}
for some $a,b,c,d,A,B,C,D\in\mathbb{R}$.
\end{theorem}


The exceptional function~(\ref{eq-th1}) has the following property: there is a $2$-parametric family of circles $S_t$, $t\in I^2$, such that for each $t\in I^2$  the restriction $F\left|_{S_t\cap U}\right.$ is a restriction of a linear function.

\bigskip
\noindent\emph{Plan of the proof of Pencil Theorem~\ref{th1}.}
We say that two circles \emph{cross} each other if their intersection consists of $2$ points. Assume that the family of circles is not a pencil. Then it contains a subfamily
of one of the following types:
\begin{itemize}
\item[(1)] the circles $S_t$, $t\in I$, pairwise cross but do not pass through one point;
\item[(2)] the circles $S_t$, $t\in I$, have a common point $O$;
\item[(3)] the circles $S_t$, $t\in I$, are nested.
\end{itemize}

First we prove the theorem in case when the region $U$ is sufficiently large, i.~e., $U\supset \bigcup S_t$, $U\supset \bigcup S_t-\{O\}$ and $U=\mathbb{R}^2$ for types (1), (2) and (3), respectively.

Then we reduce the theorem to the latter case by a biharmonic continuation of the function $F$; see Figure~\ref{figure-continuation}. The continuation is in $2$ steps.

In the first step we extend the function $F$ \emph{along} the circles $S_t$ until we reach the envelope of the family of circles (if the envelope is nonempty). This is done easily for arbitrary real analytic function $F$. The main difficulty is that extending $F$ along the circles beyond the envelope may lead to a multi-valued function.

In the second step we extend the function $F$ \emph{across} the circles $S_t$ to make the region
$U$ sufficiently large keeping the function single-valued. This is done using a new
symmetry principle for biharmonic functions.

\begin{figure}[htbp]
\centering
\definecolor{ffqqqq}{rgb}{1,0,0}
\definecolor{ffffff}{rgb}{1,1,1}
\definecolor{ffffzz}{rgb}{1,1,0.6}
\begin{tikzpicture}[line cap=round,line join=round,>=triangle 45,x=0.6cm,y=0.6cm]
\clip(-1.18,-3.77) rectangle (12.17,5.14);
\fill[color=ffffzz,fill=ffffzz,fill opacity=0.25] (-1,5) -- (12,5) -- (12,-3.53) -- (-1,-3.53) -- cycle;
\fill[line width=0pt,color=ffffff,fill=ffffff,fill opacity=1.0] (4.7,3.41) -- (9.87,0.45) -- (9.44,-2.45) -- (3.13,-2.38) -- cycle;
\draw (-1,5)-- (12,5);
\draw (12,5)-- (12,-3.53);
\draw (12,-3.53)-- (-1,-3.53);
\draw (-1,-3.53)-- (-1,5);
\draw [color=ffffff,fill=ffffff,fill opacity=1.0] (3.16,0.72) circle (1.86cm);
\draw [line width=1.2pt,color=ffqqqq] (4.7,3.41)-- (9.87,0.45);
\draw [line width=1.2pt,color=ffqqqq] (3.13,-2.38)-- (9.08,-2.45);
\draw [shift={(3.16,0.72)},line width=1.2pt,color=ffqqqq]  plot[domain=1.05:4.7,variable=\t]({1*3.1*cos(\t r)+0*3.1*sin(\t r)},{0*3.1*cos(\t r)+1*3.1*sin(\t r)});
\draw [shift={(6.33,2.48)},line width=1.2pt,dash pattern=on 5pt off 5pt]  plot[domain=-0.66:2.8,variable=\t]({1*0.83*cos(\t r)+0*0.83*sin(\t r)},{0*0.83*cos(\t r)+1*0.83*sin(\t r)});
\draw [shift={(5.07,0.2)},line width=1.2pt]  plot[domain=1.07:4.69,variable=\t]({1*2.6*cos(\t r)+0*2.6*sin(\t r)},{0*2.6*cos(\t r)+1*2.6*sin(\t r)});
\draw [shift={(5.07,0.2)},line width=1.2pt,dash pattern=on 5pt off 5pt]  plot[domain=-1.6:1.07,variable=\t]({1*2.6*cos(\t r)+0*2.6*sin(\t r)},{0*2.6*cos(\t r)+1*2.6*sin(\t r)});
\draw [shift={(9.28,-0.95)},line width=1.2pt,dash pattern=on 5pt off 5pt,color=ffqqqq]  plot[domain=-1.7:1.17,variable=\t]({1*1.51*cos(\t r)+0*1.51*sin(\t r)},{0*1.51*cos(\t r)+1*1.51*sin(\t r)});
\draw [shift={(6.33,2.48)},line width=0.4pt,fill=black,fill opacity=0.25]  (0,0) --  plot[domain=2.8:5.62,variable=\t]({1*0.83*cos(\t r)+0*0.83*sin(\t r)},{0*0.83*cos(\t r)+1*0.83*sin(\t r)}) -- cycle ;
\draw [shift={(6.33,2.48)},line width=1.2pt]  plot[domain=2.8:5.62,variable=\t]({1*0.83*cos(\t r)+0*0.83*sin(\t r)},{0*0.83*cos(\t r)+1*0.83*sin(\t r)});
\draw [shift={(9.28,-0.95)},color=ffffzz,fill=ffffzz,fill opacity=0.25]  plot[domain=1.17:4.82,variable=\t]({1*1.51*cos(\t r)+0*1.51*sin(\t r)},{0*1.51*cos(\t r)+1*1.51*sin(\t r)});
\draw [shift={(9.28,-0.95)},line width=1.2pt,color=ffqqqq]  plot[domain=1.17:4.82,variable=\t]({1*1.51*cos(\t r)+0*1.51*sin(\t r)},{0*1.51*cos(\t r)+1*1.51*sin(\t r)});
\draw (5.66,2.68) node[anchor=north west] {$V_p$};
\draw (1.52,3.52) node[anchor=north west] {$R_q$};
\draw (8.15,2.32) node[anchor=north west] {$E_0$};
\draw (8.03,0.1) node[anchor=north west] {$R_r$};
\draw (8.3,-0.92) node[anchor=north west] {$S_r-R_r$};
\draw (5.48,-2.42) node[anchor=north west] {$E-E_0$};
\fill [color=black] (6.33,2.48) circle (1.5pt);
\draw[color=black] (3.98,1.81) node {$R_p$};
\draw[color=black] (6.35,-0.53) node {$S_p-R_p$};
\end{tikzpicture}
\caption{Biharmonic continuation of a function $F$ whose restriction to an arc of each circle $S_t$, $t\in I$, is linear. First we extend $F$ along the circles to the white region bounded by certain arcs $R_q\subset S_q$, $R_r\subset S_r$ and two pieces of the envelope $E$. Then we extend $F$ across the circles to a neighborhood of the piece $E_0$ of the envelope, using reflection of the gray region $V_p$ with respect a circle $S_p$. For details refer to Lemma~\ref{l6} and its proof.}
\label{figure-continuation}
\end{figure}
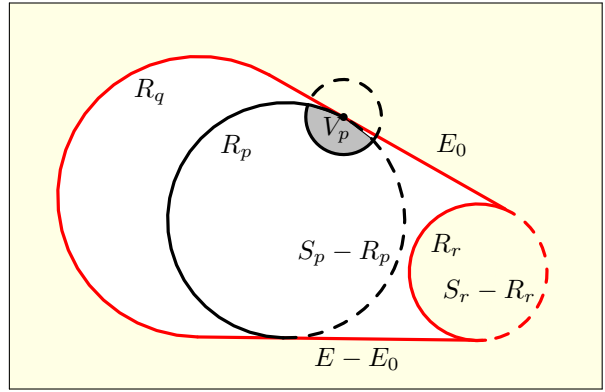

\subsection{Three typical cases}

First let us prove Theorem~\ref{th1} in three typical particular cases
treated in Lemmas~\ref{l1}, \ref{l2} and~\ref{l3} for ``sufficiently large'' sets $U$.

\begin{lemma}[Crossing circles] \label{l1} Let $S_t$, $t\in I$, be a family of pairwise crossing circles in the plane distinct from a pencil of circles. Let $F$ be an arbitrary function defined in the set $U=\bigcup_{t\in I} S_t$. Suppose that for each $t\in I$ the restriction
$F\left|_{S_t}\right.$ is a restriction of a linear function. Then $F=A((x-a)^2+(y-b)^2)+B$ for some $a,b,A,B\in\mathbb{R}$.
\end{lemma}

\begin{proof}
Denote by $l_t$ the linear function $F\left|_{S_t}\right.$.
Let $s_t=0$ be the normalized equation of the circle $S_t$, i.~e., $s_t=x^2+y^2+\dots$ and $s_1\left|_{S_1}\right.=0$. For any pair $s,t\in I$ both differences $s_t-s_s$ and $l_t-l_s$ are linear functions vanishing on $S_s\cap S_t$. Thus $l_t-l_s=k_{st}(s_t-s_s)$ for some number $k_{st}$.

Since the family $S_t$ is not a pencil it follows that there are $3$ circles $S_1,S_2,S_3$ in the family such that the functions $s_1,s_2,s_3$ are linearly independent. Let us show that $F=k_{12}s_1+l_1$.

Indeed, in the circle $S_1$ we have $F=l_1=k_{12}s_1+l_1$ because $s_1\left|_{S_1}\right.=0$. In the circle $S_2$ we have
$F=l_2=k_{12}s_2+l_2=k_{12}s_1+l_1$ by definition of the number $k_{12}$.

Consider the circle $S_3$. We have $k_{23}=k_{31}=k_{12}$ because otherwise $k_{12}(s_1-s_2)+k_{23}(s_2-s_3)+k_{31}(s_3-s_1)=
(l_1-l_2)+(l_2-l_3)+(l_3-l_1)=0$ is a nontrivial linear combination of $s_1,s_2,s_3$. Thus in the circle $S_3$ we have
$F=l_3=k_{13}s_3+l_3=k_{13}s_1+l_1=k_{12}s_1+l_1$.

Finally, take any circle $S_t$. We can replace one of the functions $s_1,s_2,s_3$ by $s_t$ to get still a linearly independent triple.
Repeating the argument from the previous paragraph we get $F=k_{12}s_1+l_1$ in $S_t$. Thus $F=k_{12}s_1+l_1$ in the whole set $U$.
\qed\end{proof}

\begin{lemma}[Circles with a common point] \label{l2} Let $S_t$, $t\in I$, be a family of pairwise crossing circles in the plane passing through the origin $O$. Assume that no three circles of the family belong to one pencil. Let $F$ be an arbitrary function defined in the set $U=\bigcup_{t\in I} S_t-\{O\}$. Suppose that for each $t\in I$ the restriction
$F\left|_{S_t-\{ O \}}\right.$ is a restriction of a linear function. Then
$$
F(x,y)=A((x-a)^2+(y-b)^2)+\frac{Bx^2+Cxy+Dy^2}{x^2+y^2}
$$
for some $a,b,A,B,C,D\in\mathbb{R}$.
\end{lemma}

\begin{proof}
Perform the transformation $(x,y,z)\mapsto(x,y,z)/(x^2+y^2)$.
Then the family of circles $S_t$ transforms to a family of lines $L_t$.
By the assumptions of the lemma any two of the lines $L_t$ intersect each other but
no three of the lines $L_t$ pass through one point.
The graph of the function $F$ transforms to a graph of a function $G$ defined in $V=\bigcup_{t\in I} L_t$.
For each $t\in I$ the restriction $G\left|_{L_t}\right.$ is a quadratic function.

Take three lines $L_1,L_2,L_3$ from the family. Let $l_1,l_2,l_3$ be nonzero linear functions vanishing in the lines $L_1,L_2,L_3$, respectively. Let $l$ be a linear function such that $l=F$ in the points $L_1\cap L_2,L_2\cap L_3, L_3\cap L_1$.
Since $G\left|_{L_1}\right.$ is quadratic and $G-l=0$ in the points $L_1\cap L_2$ and $L_1\cap L_3$ it follows that
$G\left|_{L_1}\right.=k_{23}l_2l_3+l$ for some number $k_{23}$.
Analogously, $G\left|_{L_2}\right.=k_{31}l_3l_1+l$ and $G\left|_{L_3}\right.=k_{12}l_1l_2+l$ for some numbers $k_{12}$ and~$k_{31}$.

Let us prove that $G=k_{12}l_1l_2+k_{23}l_2l_3+k_{31}l_3l_1+l$ in the whole set $V$. Indeed, consider the difference
$H=k_{12}l_1l_2+k_{23}l_2l_3+k_{31}l_3l_1+l-G$. Then $H\left|_{L_1}\right.=0$, $H\left|_{L_2}\right.=0$, $H\left|_{L_3}\right.=0$ by the above.
Take a line $L_t$ distinct from $L_1,L_2,L_3$. Then $H\left|_{L_t}\right.$ is a quadratic function.
On the other hand, $H(L_t\cap L_1)=H(L_t\cap L_2)=H(L_t\cap L_3)=0$. Since the points
$L_t\cap L_1,L_t\cap L_2,L_t\cap L_3$ are pairwise distinct it follows that $H\left|_{L_t}\right.=0$. So the function $H$ vanishes in each line $L_t$. Thus $H=0$ in the set $V$.

We have proved that $G$ is a polynomial of degree not greater than $2$. Performing the inverse transformation $(x,y,z)\mapsto(x,y,z)/(x^2+y^2)$ we obtain the required formula for the function $F$.
\qed\end{proof}

\begin{lemma}[Nested circles] \label{l3}
Let $S_1$ and $S_2$ be the pair of circles $x^2+y^2=1$ and $x^2+y^2=2$.
Let $F$ be a function biharmonic in the whole plane $\mathbb{R}^2$.
Suppose that for each $t=1,2$ the restriction
$F\left|_{S_t}\right.$ is a restriction of a linear function. Then
$$
F(x,y)=(x^2+y^2)(Ax+By+C)+ax+by+c
$$
for some $a,b,c,A,B,C\in\mathbb{R}$.
\end{lemma}

The function $F(x,y)=(x^2+y^2)\log(x^2+y^2)$ extended by $F(0,0)=0$ might seem to be a counter-example to this lemma but in fact it is not:
$\partial^2 F/\partial x^2$ is discontinuous at the origin.

\begin{proof}
By Proposition~\ref{cl1} below it follows that there are functions $u_1,u_2$, harmonic in $\mathbb{R}^2$,
such that $F=(x^2+y^2-2)u_1+u_2$. Then
$F=(x^2+y^2-2)u_3+(x^2+y^2-1)u_2$, where the function $u_3=u_1-u_2$ is also harmonic in $\mathbb{R}^2$.
Since $u_2$ is harmonic in the whole plane $\mathbb{R}^2$ and the restriction
$u_2\left|_{S_2}\right.=F\left|_{S_2}\right.$ is linear it follows by uniqueness theorem that the function $u_2$ itself is linear.
Analogously $u_3$ is linear and the lemma follows.
%
%
%
\qed\end{proof}

\begin{proposition}\label{cl5} Let
$F(x,y)=(x^2+y^2)(Ax+By+C)+ax+by+c$, where $A^2+B^2\ne 0$. Suppose that the restriction of the function $F$ to a circle $S\subset\mathbb{R}^2$ is linear. Then the center of the circle $S$ is the origin.
\end{proposition}

\begin{proof}
It suffices to consider the case when $a=b=c=C=0$.
Let $x^2+y^2+px+qy+r=0$ be the equation of $S$.
If the restriction $F\left|_S\right.$ is linear then $S$ is a projection of (a part of)
the intersection of the surface $z=F(x,y)$ and a plane $z=\alpha x+\beta y+\gamma $. Thus there exist numbers $k,l,m$ such that
$$
(x^2+y^2)(Ax+By)-(\alpha x+\beta y+\gamma)=
(kx+ly+m)(x^2+y^2+px+qy+r).
$$
We get
$$
k=A, \quad l=B, \quad pk+m=0, \quad ql+m=0, \quad pl+qk=0.
$$
Thus $p=-An$, $q=Bn$, $m=A^2n=-B^2n$ for some $n\in\mathbb{R}$. Since $A^2+B^2\ne 0$ it follows that $n=0$. So $S$ is the circle $x^2+y^2+r=0$.
\qed\end{proof}

\subsection{Biharmonic continuation}

We are going to reduce Theorem~\ref{th1} to Lemmas~\ref{l1}, \ref{l2} and~\ref{l3} by ``biharmonic continuation'' of the function $F$.
We say that a function $F$ biharmonic in a region $U$ \emph{extends} to a function $G$ biharmonic in a region $V$ if there is an open set $D\subset U\cap V$ such that $F=G$ in $D$. Notice that $F$ can be distinct from $G$ in $U\cap V$ if the latter set is disconnected.

\begin{proposition}[Uniqueness of a continuation] \label{c0} If two functions biharmonic in a region $V\subset\mathbb{R}^2$ coincide in a region $U\subset V$ then these functions coincide in the region $V$.
\end{proposition}

\begin{proof} Let $F,G$ be functions such that $\Delta^2 F=\Delta^2 G=0$ in $V$ and $F=G$ in $U$. Then $\Delta(F-G)$ is harmonic in $V$ and vanishes in $U$. Thus $\Delta(F-G)=0$ in $V$. Hence $F-G$ is harmonic in $V$ and vanishes in $U$. Thus $F-G=0$ in~$V$.
\qed\end{proof}


\mscomm{Recall that the \emph{envelope} of a family of curves $s(x,y,t)=0$ is the set of all $(x,y)\in\mathbb{R}^2$ such that there exists $t\in I$ such that $\tfrac{\partial}{\partial t}s(x,y,t)=s(x,y,t)=0$. The envelope can be empty.}

\begin{lemma}[Continuation along circles] \label{l4} Let $F:U\to\mathbb{R}$ be a biharmonic function defined in a region $U\subset\mathbb{R}^2$.
Let $S_t$, $t\in I$, be an analytic family of circles in the plane
containing at least two distinct ones.
Suppose that for each $t\in I$ we have $S_t\cap U\ne\emptyset$ and the restriction $F\left|_{S_t\cap U}\right.$ is a restriction of a linear function. Then for some segment $J=[q,r]\subset I$ the function $F$ extends to a function biharmonic in a region bounded by certain arcs of the circles $S_{q}$, $S_{r}$ and possibly two pieces of the envelope of the family $S_t$, $t\in J$.
\end{lemma}

\begin{proof}
The idea of the proof is to extend the function linearly along the circles until we reach the envelope. \mscomm{For that we show that the envelope cuts a family of disjoint arcs on the circles. The resulting} function will be a real analytic continuation of the initial function and hence it will be biharmonic. \mscomm{Hereafter all assertions hold for some segment $J\subset I$ possibly smaller than $I$.}

\mscomm{Let us recall general structure of the envelope of an analytic family. For an appropriate segment $J_1\subset I$, each connected component of the envelope $E$ of the family $S_t$, $t\in J_1$, is either an isolated point common to all $S_t$ or a regular analytic curve tangent to all $S_t$ at pairwise distinct points and having no other common points with $S_t$. Indeed, let $s(x,y,t):=x^2+y^2+a(t)x+b(t)y+c(t)=0$ be the equation of the circle $S_t$ (the argument does not actually rely on the particular form of the equation). For each $n=1,2,\dots$ consider the subset $\Sigma_n$ of $\mathbb{R}^2\times I$ given by
$s(x,y,t)=\tfrac{\partial}{\partial t}s(x,y,t)=\dots=\tfrac{\partial^n}{\partial t^n}s(x,y,t)=0$ but $\tfrac{\partial^{n+1}}{\partial t^{n+1}}s(x,y,t)\ne 0$, and also the subset $\Sigma_\infty$ given by $\tfrac{\partial^n}{\partial t^n}s(x,y,t)=0$ for all $n$. The projection of the union $\Sigma$ of all these subsets onto the $xy$-plane is the envelope $E$. By the analyticity, we may restrict to a subsegment $J_1\subset I$ so that each of these subsets is either empty or a collection of the whole connected components of $\Sigma$. Moreover, we may assume that the connected components are not contained in the planes of the form $\mathbb{R}^2\times \{t\}$, and their projections are disjoint. First take a point $(x,y,t)\in\Sigma_\infty$. Then by the analyticity $s(x,y,t)=0$ for all $t$. Hence $(x,y)$ is a point of the envelope common to all $S_t$ (which we may assume to be isolated). Now take a point $(x,y,t)\in\Sigma_n$ and consider the map $G(x,y,t):= \left(s(x,y,t),\tfrac{\partial^n}{\partial t^n}s(x,y,t)\right)$; this generalizes the argument  from~\cite[\S5.21]{bruce-giblin:1984}, where $n=1$. Then the connected component of the envelope containing $(x,y)$ is a subset of the projection of $G^{-1}(0,0)$ into the $xy$-plane. Since $\tfrac{\partial s}{\partial t}=0$, $\tfrac{\partial^{n+1} s}{\partial t^{n+1}}\ne 0$, and $\left(\tfrac{\partial s}{\partial x},\tfrac{\partial s}{\partial y}\right)\ne (0,0)$, it follows that the differential $dG$ is surjective. Then by the Implicit Function Theorem, the intersection of $G^{-1}(0,0)$ with a neighborhood of $(x,y,t)$ is a regular analytic curve with the tangential direction $(dx,dy,dt)$  given by $\tfrac{\partial s}{\partial t}dt+\tfrac{\partial s}{\partial x}dx+\tfrac{\partial s}{\partial y}dy=\tfrac{\partial^{n+1} s}{\partial t^{n+1}}dt+\tfrac{\partial^{n+1} s}{\partial t^{n}\partial x}dx+\tfrac{\partial^{n+1} s}{\partial t^{n}\partial y}dy=0$; cf.~\cite[Proof of Proposition~5.25]{bruce-giblin:1984}. Since $\tfrac{\partial s}{\partial t}=0$ and $\tfrac{\partial^{n+1}s}{\partial t^{n+1}}\ne 0$, it follows that the projection of the curve is a regular analytic curve tangent to $S_t$. Since no component of $\Sigma_n$ is contained in the plane $\mathbb{R}^2\times \{t\}$, it follows that the projection is tangent to $S_t$ for each  $t\in J_1$ and has no other common points with $S_t$, for sufficiently small $J_1\subset I$.}

\mscomm{By the analyticity, there is $J_2\subset J_1$ such that each connected component of the envelope $E$ (not an isolated point) has constant contact order $n$ (which may depend on the component and need not equal the above number $n$) with $S_t$ for each $t\in J_2$. We have $n=1$ or $n=2$, because if a curve had contact of order $\ge 3$ with a circle $S_t$ at each point, then it would have constant curvature and all $S_t$ would coincide.}


\mscomm{
There is $J_3\subset J_2$ such that each circle $S_t$, $t\in J_3$, has at most $2$ common points with} the envelope $E$, depending on the arrangement of the circles sufficiently close to $S_t$.
\mscomm{Indeed, on the envelope, we have
$s(x,y,t):=x^2+y^2+a(t)x+b(t)y+c(t)=0$ and
$\tfrac{\partial}{\partial t}s(x,y,t)=a'(t)x+b'(t)y+c'(t)=0$. For fixed $t$, the two equations have at most two common solutions, unless $a'(t)=b'(t)=c'(t)=0$. This cannot hold identically for each $t$ because all the circles $S_t$ coincide otherwise. Thus by the analyticity there is smaller interval $J_3\subset J_1$ where at least one of the derivatives $a'(t),b'(t),c'(t)$ does not vanish, hence each circle $S_t$ has at most $2$ common points with $E$.}

Let $R_t\subset S_t$ be one of the open arcs joining the \mscomm{common} points of the circle $S_t$ and the envelope $E$. Let $R_t$ be one of the sets $S_t-O_t$ or $\emptyset$ (respectively, $R_t=S_t$ or $\emptyset$), if there is a unique such \mscomm{common} point $O_t$ (respectively, no \mscomm{common} points). Choose the arcs $R_t$ so that they form a continuous family.

One can assume that for each $t$ in a segment \mscomm{$J_4\subset J_3$} we have $R_t\cap U\ne \emptyset$. Indeed, if $R_t\cap U\ne \emptyset$ for at least one \mscomm{$t\in J_3$} then the same condition holds in a neighborhood \mscomm{$J_4$} of $t$. Otherwise replace each $R_t$ by $S_t-\mathrm{Cl}\,R_t$ and repeat the argument.

\mscomm{Then} there is a segment \mscomm{$[q,r]=J\subset J_4$} such that the arcs $R_t$, $t\in J$, are pairwise disjoint. \mscomm{Indeed, for contact order $n=2$ this follows from the Tait--Kneser theorem: all $S_t$ are disjoint as osculating circles of a plane curve $E$. For $n=1$, the circles $S_t$ are locally from one side with respect to the envelope $E$, hence the arcs $R_t$ have no intersection points in a neighborhood of the envelope, and by \cite[\S5.9]{bruce-giblin:1984} they have no intersection points outside the neighborhood for sufficiently small $J\subset I$.} 

So $V=\bigcup_{t\in J}R_t$ is a region bounded by the arcs $R_{q}$, $R_{r}$ and possibly two pieces of the envelope $E$.

By the assumption of the lemma the restriction $F\left|_{R_t\cap U}\right.$ is the restriction of a linear function for each $t\in J$. Extend the function $F$ linearly to each arc $R_t$. We get a function defined in the whole region $V$. It remains to prove that the obtained function is biharmonic in $V$.

Let us show that $F$ is real analytic in $V$.
Parametrize the arc $R_t$ by the functions $x(t,\phi)=x_0(t)+r(t)\cos\phi$, $y(t,\phi)=y_0(t)+r(t)\sin\phi$.
Consider $(t,\phi)$ as coordinates in $V$. Since the family $S_t$ is analytic, \mscomm{by \cite[Proposition~5.20]{bruce-giblin:1984}} it follows that these coordinates are analytic.
Without loss of generality assume $[q,r]\times [\alpha,\beta]\subset U$ for some $\alpha,\beta\in[-\pi,\pi]$. Then $F(t,\phi)$ is real analytic in $[q,r]\times [\alpha,\beta]$.
By the construction $F(t,\phi)=a(t)\cos\phi+b(t)\sin\phi+c(t)$ in the region $V$ for some functions $a(t),b(t),c(t)$.
Thus $a(t),b(t), c(t)$ are real analytic in $[q,r]$. Hence $F$ is real analytic in the whole region $V$.

Then the function $\Delta^2 F$ is also real analytic in the region $V$ and vanishes in the open set $U\cap V$.
By the uniqueness theorem for analytic functions it follows that $\Delta^2 F=0$ in the whole region $V$, i.~e., $F$ is biharmonic in $V$.
\qed\end{proof}

To extend the function $F$ further we need the following preparations.

\begin{proposition}[Representation] \label{cl1} \textup{\cite{balk:1970}} Let $s(x,y)=x^2+y^2+ax+by+c$. Then any function $F$ biharmonic in a simply-connected region $U\subset\mathbb{R}^2$ can be represented as  $F=su_1+u_2$
for some functions $u_1,u_2$ harmonic in~$U$.
\end{proposition}




\begin{proposition}[Arc extension] \label{cl2} Let $S\subset R$ be a pair of circular arcs.
Let $F$ be a biharmonic function defined in a neighborhood of the arc $R$. Suppose that
$F\left|_{S}\right.$ is a restriction of a linear function. Then $F\left|_{R}\right.$
is the restriction of the same linear function.
\end{proposition}

\begin{proof} Let $l$ be the linear function $F\left|_{S}\right.$.
Let $s(x,y)=0$ be the normalized equation of the circle containing the arc $S$.
By Proposition~\ref{cl1} we have $F-l=su_1+u_2$ for some functions $u_1,u_2$ harmonic in a neighborhood $U$ of the arc $R$.
Then $u_2\left|_{S}\right.=(F-l)\left|_{S}\right.=0$. By the symmetry principle for harmonic functions it follows that
$u_2(x,y)=-u_2(x',y')$ for any pair of points $(x,y),(x',y')\in U$ symmetric with respect to the circle $s(x,y)=0$.
In particular, $u_2\left|_{R}\right.=0$. Thus $F\left|_{R}\right.=l$.
\qed\end{proof}

Now we are going to give a version of a symmetry principle for biharmonic functions. The usual symmetry principle \cite{poritzki:1946,duffin:1955} is not applicable in our situation because we have no information on the growth of the function in the normal directions to the circles.

We use the following notation. Let $S_s$ and $S_t$ be a pair of circles.
Denote by $O_t$ the center of the circle $S_t$ and by $r_t:\mathbb{R}^2-\{ O_t\}\to \mathbb{R}^2-\{ O_t\}$ the reflection with respect to the circle $S_t$. Let $r_t(U)$ be a shorthand for $r_t(U-\{O_t\})$. Denote by $\Sigma_{st}$ the \emph{limit set} of the pencil of circles passing through $S_s$ and $S_t$, i.~e.,
$\Sigma_{st}=\{\,x\in\mathbb{R}^2\,:\, r_s(x)=r_t(x) \, \}$. If $S_s\ne S_t$ then the limit set consists of at most $2$ points.

The following technical definition is required to keep the function single-valued
during the continuation process with our symmetry principle.

\begin{definition}[Nicely arranged region] \label{def-nicelyaranged}
A region $U\subset\mathbb{R}^2$ is \emph{nicely arranged} with respect to two circles $S_1$ and $S_2$, if the set $U\cap r_1(U)\cap r_2(U)-\Sigma_{12}$ has a connected component $D$ such that $S_1\cap D\ne\emptyset$ and $S_2\cap D\ne\emptyset$.
\end{definition}

Notice that an arbitrary region is nicely arranged with respect to any pair of sufficiently close circles intersecting the region.

\begin{lemma}[Double symmetry principle] \label{cl4} Let $F$ be a function biharmonic in a simply-connected region $U\subset\mathbb{R}^2$ nicely arranged with respect to a pair of circles $S_1\ne S_2$. Suppose that for each $t=1,2$
the restriction $F\left|_{S_t\cap U}\right.$ is a restriction of a linear function.
Then $F$ extends to a
function biharmonic in the open set $r_1(U)\cap r_2(U)-\Sigma_{12}$. 
\end{lemma}

\begin{proof} Let $l_t$ be the linear function $F\left|_{S_t}\right.$ for $t=1,2$.
Without loss of generality assume that $S_1$ is the unit circle $x^2+y^2=1$.
By Proposition~\ref{cl1} it follows that $F=(x^2+y^2-1)u_2+u_1+l_1$ for some functions $u_1$ and $u_2$ harmonic in~$U$.

Take functions $\nu_1(z)$ and $\nu_2(z)$ complex analytic in $U$ such that $u_t=\nu_t(z)+\overline{\nu_t(z)}$ for $t=1,2$. Since $U$ is simply-connected it follows that $\nu_1(z)$ and $\nu_2(z)$ are single-valued.
Let $\lambda_1(z)$ and $\lambda_2(z)$ be linear functions such that $l_t=\lambda_t(z)+\overline{\lambda_t(z)}$ for $t=1,2$.
For $t=1,2$ represent the reflection with respect to the circle $S_t$ as a map $z\mapsto \rho_t(\bar z)$
for a fractional linear function $\rho_t(z)$.


Let us extend the function $u_1$ to the open set $r_1(U)$. (What we do is the usual symmetry principle.)
Let $D$ be the open set from Definition~\ref{def-nicelyaranged}.
For each $z\in S_1$ we have $z=\rho_1(\bar z)$. Thus the condition
$F\left|_{S_1}\right.=l_1$ is equivalent to
\begin{equation}\label{simple}
\nu_1(z)=-\overline{\nu_1(\rho_1(\bar z))}
\end{equation}
for each $z\in S_1\cap D$. Both sides of formula~(\ref{simple}) are complex analytic functions in $D$. By the uniqueness theorem it follows that these functions coincide in $D$.
Thus formula~(\ref{simple}) defines an extension of the function $\nu_1(z)$ to the open set $r_1(U)$.
So $u_1=\nu_1(z)+\overline{\nu_1(z)}$ is the required extension of the function $u_1$.

Let us extend the function $u_2$ to the open set $r_1(U)\cap r_2(U)-\Sigma_{12}$.
For each $z\in S_2$ we have $z=\rho_2(\bar z)$.
For each $z\in D$ formula~(\ref{simple}) holds by the previous paragraph.
Thus for each $z\in S_2\cap D$ the condition $F\left|_{S_2}\right.=l_2$ is equivalent to the condition
\begin{multline}\label{formula}
\nu_2(z)=-\overline{\nu_2(\rho_2(\bar z))} +
\frac{\overline{\nu_1(\rho_1(\bar z))}-\overline{\nu_1(\rho_2(\bar z))}
-\lambda_1(z)-\overline{\lambda_1(\rho_2(\bar z))}+\lambda_2(z)+\overline{\lambda_2(\rho_2(\bar z))}}
{z\overline {\rho_2(\bar z)}-1}.
\end{multline}
Since both sides of formula~(\ref{formula}) are complex analytic functions in $D$ it follows that these functions coincide in $D$. If $z\in r_1(U)\cap r_2(U)$ then $\rho_1(\bar z),\rho_2(\bar z)\in U$. Thus the right-hand side of formula~(\ref{formula}) defines a function complex analytic in $r_1(U)\cap r_2(U)-\Sigma_{12}$ because the denominator may vanish only in $\Sigma_{12}$. Extend the function $\nu_2(z)$ to the open set $r_1(U)\cap r_2(U)-\Sigma_{12}$ by formula~(\ref{formula}).
Then $u_2(z)=\nu_2(z)+\overline{\nu_2(z)}$ is the required extension of the function $u_2$.

Since both functions $u_1$ and $u_2$ extend to $r_1(U)\cap r_2(U)-\Sigma_{12}$ it follows that
$F=(x^2+y^2-1)u_2+u_1+l_1$ also extends to $r_1(U)\cap r_2(U)-\Sigma_{12}$.
\qed\end{proof}

\begin{lemma}[Continuation across nested circles] \label{l5}
Let $S_t$, $t\in I$, be a family of nested circles in the plane distinct from a pencil of circles.
Let $F:U\to\mathbb{R}$ be a function biharmonic in the ring $U$ between $S_0$ and $S_1$.
Suppose that for each $t\in I$ the restriction $F\left|_{S_t\cap U}\right.$ is a restriction of a linear function. Then the function $F$ extends to a function biharmonic in the whole plane $\mathbb{R}^2$.
\end{lemma}

\begin{proof} The idea of the proof is to extend the function, using the double symmetry principle,
to the ring between $r_0(S_1)$ and $S_1$, then to the ring between $r_0(S_1)$ and $r_1r_0(S_1)$, and so on.

Take any pair of circles $S_t$ and $S_s$, where $s,t\in I$ are sufficiently close to $0$.
Draw disjoint slits $T$ and $T'$ such that the regions $U-T$ and $U-T'$ are simply-connected.
By Lemma~\ref{cl4} the function $F$ extends to both $r_t(U-T)\cap r_s(U-T)$ and $r_t(U-T')\cap r_s(U-T')$.
Thus it extends to a (possibly multi-valued) function biharmonic in the ring $r_t(U)\cap r_s(U)\cup U$.
The latter function is single-valued because a continuation along the closed path $S_0$ leads to the initial value. Approaching $t,s\to 0$ one can extend the function $F$ to the ring between the circles $r_0(S_1)$ and $S_1$. Now approaching $t,s\to 1$ one can extend the function $F$ to the larger ring between the circles $r_0(S_1)$ and $r_1r_0(S_1)$. Continuing this process one extends $F$ to a function biharmonic in $\mathbb{R}^2$ except the limit set $\Sigma_{01}$ of the pencil of circles passing through $S_0$ and $S_1$.

Since $S_t$, $t\in I$, is not a pencil of circles it follows that $\Sigma_{0p}\cap\Sigma_{01}=\emptyset$ for some $p\in I$. Repeating the above reflection process for the pair of circles $S_0$ and $S_p$ one extends the function $F$ to a function biharmonic in the whole plane~$\mathbb{R}^2$.
\qed\end{proof}

\mscomm{Notice that even nested circles might have nonempty envelope: e.g., for the family of osculating circles of an analytic curve, the envelope contains the curve itself.}

\begin{lemma}[Continuation \mscomm{across circles beyond the envelope}] \label{l6}
Let $S_t$, $t\in I$, be an analytic family of 
circles in the plane distinct from a pencil.
Let $F:U\to\mathbb{R}$ be a function biharmonic in a region
$U\subset\mathbb{R}^2$.
Suppose that for each $t\in I$ we have $S_t\cap U\ne\emptyset$ and the restriction $F\left|_{S_t}\right.$ is a restriction of a linear function.
Then for some segment $J\subset I$ the function $F$ extends to a function biharmonic in a neighborhood of
$\bigcup_{t\in J} S_t$ possibly except a common point of all the circles $S_t$.
\end{lemma}

\begin{proof} The idea of the proof is to extend the function first
along the circles until we reach the envelope $E$ of the family, then by the double symmetry principle --- to a neighborhood of the envelope, and finally --- along the circles beyond the envelope; see Figure~\ref{figure-continuation}.

By Lemma~\ref{l4} it follows that
$F$ extends to a region bounded by certain arcs of the circles $S_q$ and $S_r$ and \mscomm{possibly} two pieces of the envelope $E$ for some $q,r\in I$. \mscomm{This completes the proof, if $E=\emptyset$. Otherwise,} at least one component $E_0$ of the envelope $E$ does not degenerate to a point, because the family $S_t$, $t\in [q,r]$, is not a pencil.

Let us extend the function $F$ to a neighborhood of the curve $E_0$.
Use the notation from the proof of Lemma~\ref{l4}.
Take $p\in [q,r]$. 
Let $V_p$ be the intersection of the open disc bounded by the circle $S_p$, an open disc of centered at \mscomm{$\mathrm{Cl}R_p\cap E_0$, and the region $V$.} 
Without loss of generality assume that $R_t\cap V_p=\emptyset$ for each $t\in [p-\epsilon,p]$ and
$R_t\cap V_p\ne\emptyset$ for each $t\in [p,p+\epsilon]$, where $\epsilon>0$ is small enough.

Take a pair of circles $S_t$ and $S_s$, where $s,t\in [p,p+\epsilon]$. By Lemma~\ref{cl4} the function $F$ extends to the region $r_t(V_p)\cap r_s(V_p)-\Sigma_{st}$. Approaching $t,s\to p$ one extends the function $F$ 
to a neighborhood of the point $R_p\cap E_0$; \mscomm{in particular, the latter is not a singular point for $F$}. So $F$ extends to a
neighborhood \mscomm{$V'$} of the curve $E_0$.

A consequence of this extension is that \mscomm{$F\left|_{S_t\cap V'}\right.$} is linear for each $t\in I$, because
each intersection \mscomm{$S_t\cap V'$} is connected and Proposition~\ref{cl2} can be applied.

Let us extend the function $F$ along the arcs $S_t-R_t$. Choose \mscomm{$U'\subset V'$} and $I'\subset I$ so that $R_t\cap U'=\emptyset$ and $S_t-R_t\cap U'\ne\emptyset$ for each $t\in I'$. Applying Lemma~\ref{l4} to the region $U'$ and the family $S_t$, $t\in I'$, we extend the function $F$ to a region bounded by
$S_{q'}-R_{q'}$, $S_{r'}-R_{r'}$ and certain pieces of the envelope $E$ for some $p',q'\in I'$. Take a segment $J$ strictly inside $[q',r']$. Then the function $F$ extends to a (possibly multi-valued) function biharmonic in a neighborhood of $\bigcup_{t\in J} S_t$ possibly except a common point of all the circles $S_t$, $t\in J$.
The latter function is single-valued because a continuation along any closed path $S_t$, $t\in J$, leads to the initial value.
\qed\end{proof}

\subsection{Proof and corollaries of the Pencil Theorem}

\begin{proof}[of Theorem~\ref{th1}]
Assume that $S_t$, $t\in I$, is not a pencil of circles. Clearly, there is a segment $J\subset I$ such that one of the following conditions hold:
\begin{itemize}
\item[(1)] the circles $S_t$, $t\in J$, pairwise cross but do not pass through one point;
\item[(2)] the circles $S_t$, $t\in J$, have a common point $O$ but no three circles $S_t$, $t\in J$, belong to one pencil;
\item[(3)] the circles $S_t$, $t\in J$, are nested.
\end{itemize}
Consider each case separately.

Case (1). By Lemma~\ref{l6} the function $F$ extends to a function biharmonic in a neighborhood of the set $\bigcup_{t\in J_1} S_t$ for some segment $J_1\subset J$. By Proposition~\ref{cl2} for each $t\in J_1$ the restriction $F\left|_{S_t}\right.$ is linear.
Then by Lemma~\ref{l1} case~(1) follows.

Case (2). Analogously to the previous paragraph case~(2) follows from Lemmas~\ref{l6}, ~\ref{l2} and Proposition~\ref{cl2}.

Case (3). By Lemma~\mscomm{\ref{l6}} it follows that
the function $F$ extends to the ring between a pair of circles from the family.
By Theorem~\ref{cl3} we may assume that these two circles are $x^2+y^2=1$ and $x^2+y^2=2$. Then by Lemma~\ref{l5} the function $F$ extends to the whole plane $\mathbb{R}^2$. Thus case~(3) follows from Lemma~\ref{l3} and Proposition~\ref{cl5}.
\qed\end{proof}

The following corollaries of Theorem~\ref{th1} are straightforward; see Table~\ref{isotropic}.

\begin{corollary} \label{th-pencil}
Let $\Phi^i$ be an i-Willmore surface carrying an analytic family~$\mathcal{F}^i$ of i-M-circles. Then either the surface $\Phi^i$ is i-M-equivalent to an i-paraboloid or the top view of the family~$\mathcal{F}^i$ is a pencil of circles or lines.
\end{corollary}

\begin{corollary} \label{cor-pencil} Let $\Phi$ be an L-minimal surface enveloped by an analytic family~$\mathcal{F}$ of cones. 
Then either the surface $\Phi$ is a parabolic cyclide or a sphere, or the Gaussian spherical image of the family~$\mathcal{F}$ is a pencil of circles in the unit sphere.
\end{corollary}

\begin{corollary} \label{cor-catalan} A ruled L-minimal surface is a Catalan surface, i.~e., contains a family of line segments parallel to one plane.
\end{corollary}

\begin{proof}[of Corollary~\ref{cor-catalan}] Let $\Phi$ be a ruled L-minimal surface. It suffices to prove that \mscomm{$\Phi$ contains an \emph{analytic} family of lines}; then by Corollary~\ref{cor-pencil} the result follows.
Since $\Phi$ is L-minimal by Theorem~\ref{pgm} and Proposition~\ref{cor1} it follows that $\Phi$ itself is analytic. Since $\Phi$ is ruled it follows that
for each point $r\in\Phi$ there is a line $L_r\subset\Phi$.
The direction of the line $L_r$ is an asymptotic direction of the surface $\Phi$ at the point $r$. \mscomm{By the definition of an L-minimal surface, we have $K\ne 0$, hence there are two asymptotic directions at $r$, which depend on $r$  analytically. If one of the asymptotic curves through $r$ is not a line, then all sufficiently close asymptotic curves are neither. Hence locally one can choose a line $L_r$ (among possibly two ones through $r$) so that it depends on $r$ analytically.}
\qed\end{proof}

%


\begin{remark} Theorem~\ref{th1} does not remain true for biharmonic functions $\mathbb{C}^2\to \mathbb{C}$. For instance, for the function $F(x,y)=(x^2+y^2)(x+iy)$ there is a $2$-parametric family of circles $S_t$, $t\in I^2$, such that for each $t\in I^2$ the restriction $F\left|_{S_t}\right.$ is a restriction of a linear function.
\end{remark}

\begin{remark} Theorem~\ref{th1} does not remain true for real analytic functions $\mathbb{R}^2\to \mathbb{R}$. For instance, the restriction of the function $F(x,y)=\sqrt{(x^2+y^2)^2-x^2+1}$ to each circle of the family $x^2+y^2-tx-\sqrt{t^2-1}=0$ is a restriction of a linear function.
\end{remark}

\begin{remark} The proof of Theorem~\ref{th1} is simpler in the generic case when the biharmonic function $F$ extends to a (possibly multi-valued) function in the whole plane except a discrete subset $\Sigma$. For instance, to prove Lemma~\ref{l6} in this case it suffices to take a segment $J\subset I$ such that $S_t\cap \Sigma=\emptyset$ for each $t\in J$.
\end{remark}


\section{Classification of L-minimal surfaces enveloped by a family of cones}\label{classification}

\subsection{Elliptic families of cones}

The results of the previous section give enough information to describe all the L-minimal surfaces enveloped by a family of cones, in particular, ruled L-minimal surfaces. We have got to know that either the top view of
a family of i-M-circles in an i-Willmore is a pencil
or the surface contains another family of i-circles with top view being a pencil.
Let us consider separately each possible type of the pencil.

\begin{definition} \label{def-elliptic} An \emph{elliptic} pencil of circles in the plane (or in a sphere) is the set of all the circles passing through two fixed distinct points. A $1$-parametric family of cones (possibly degenerating to cylinders or lines) in space is
\emph{elliptic} if the Gaussian spherical images of the cones form an elliptic pencil of circles in the unit sphere.
\end{definition}

Denote by $\mathrm{Arctan}\, x=\{\,\arctan x+\pi k\,:\,k\in\mathbb{Z}\,\}$ the multi-valued inverse of the tangent function.

\begin{theorem} \label{i-classification-elliptic}
Let $\Phi^i$ be an i-Willmore surface carrying a family of i-M-circles.
Suppose that the top view of the family is an elliptic pencil of circles.
Then the surface $\Phi^i$ is i-M-equivalent to a piece of the surface
\begin{multline}
z=\left(a_1(x^2+y^2)+a_2x+a_3\right)\mathrm{Arctan}\,\frac{y}{x}
+\frac{b_1y^2+b_2xy}{x^2+y^2}+c_1y^2+c_2xy\label{eq-th2-2}\\[-10pt]
\end{multline}
for some $a_1,a_2,a_3,b_1,b_2,c_1,c_2\in\mathbb{R}$.
\end{theorem}

\begin{proof}[of Theorem~\ref{i-classification-elliptic}]
Perform an i-M-transformation taking the elliptic pencil of circles in the top view to the pencil of lines $y=tx$, where $t$ runs through a segment $J\subset \mathbb{R}$.
Denote by $z=F(x,y)$ the surface obtained from the surface $\Phi^i$ by the transformation, where $F$ is a biharmonic function defined in a region $U\subset\mathbb{R}^2$. Assume without loss of generality that $(0,0)\not\in U$ and $F$ is single-valued in $U$.
Since an i-M-transformation takes i-M-circles to i-M-circles it follows that
the restriction of the function $F$ to (an appropriate segment of) each line $y=tx$, where $t\in J$, is a quadratic function.

\begin{proposition} \label{auxillary} Let $F(x,y)$ be a biharmonic function in a region $U\subset\mathbb{R}^2-\{(0,0)\}$.
Suppose that the restriction of the function $F$ to the intersection of each line $y=tx$, where $t\in J$, with the region $U$ is a quadratic function. Then
\begin{multline}
F(x,y)=\left(a_1(x^2+y^2)+a_2x+a_3+a_4y\right)\arctan\frac{y}{x}+\\
+\frac{b_1y^2+b_2xy+b_3x^2}{x^2+y^2}+c_1y^2+c_2xy+c_3x^2+d_1x+d_2y\label{eq-auxillary}
\end{multline}
for some $a_1,a_2,a_3,a_4,b_1,b_2,b_3,c_1,c_2,c_3,d_1,d_2\in\mathbb{R}$.
\end{proposition}

\begin{proof} Consider the polar coordinates in $U$.
Restrict the function $F$ to a subregion of the form $(r_1,r_2)\times (\phi_1,\phi_2)\subset U$.
Then $F(r,\phi)=a(\phi)r^2+b(\phi)r+c(\phi)$ in the region $(r_1,r_2)\times (\phi_1,\phi_2)$ for some smooth functions $a(\phi), b(\phi), c(\phi)$.
Thus $r^4\Delta^2F=\left(4a''+a^{(4)}\right)r^2+\left(b+2b''+b^{(4)}\right)r+\left(4c''+c^{(4)}\right)$.
Since $\Delta^2F=0$ it follows that the coefficients of this polynomial in $r$ vanish.
Solving the obtained ordinary differential equations we get:
\begin{align*}
a(\phi)&=\alpha_1+\alpha_2\phi+\alpha_3\cos2\phi+\alpha_4\sin2\phi;\\
b(\phi)&=\beta_1\cos\phi+\beta_2\sin\phi+\beta_3\phi\cos\phi+\beta_4\phi\sin\phi;\\
c(\phi)&=\gamma_1+\gamma_2\phi+\gamma_3\cos2\phi+\gamma_4\sin2\phi,
\end{align*}
for some $\alpha_1,\dots,\alpha_4,\beta_1,\dots,\beta_4,\gamma_1,\dots,\gamma_4\in\mathbb{R}$.
Returning to the initial \mscomm{Cartesian} coordinate system we get the 
required formula. 
\qed\end{proof}

To complete the proof of Theorem~\ref{i-classification-elliptic} perform an appropriate rotation around the $z$-axis to achieve $a_4=0$ in formula~(\ref{eq-auxillary}) and then the
i-M-transformation $z\mapsto z-c_3(x^2+y^2)-d_1x-d_2y-b_3$
to achieve $b_3=c_3=d_1=d_2=0$.
\qed\end{proof}

\begin{table}[htbp]
\caption{Biharmonic functions whose restrictions to each line $y=tx$, $t\in I$, are quadratic functions and corresponding Laguerre minimal surfaces (or a Legendre surface in case of $\mathbf{r}_2$)}
\label{trans-elliptic}
\begin{tabular}{|l|r|}
\hline
Biharmonic function & Laguerre minimal surface\\
\hline
$(x^2+y^2-1)\mathrm{Arctan}(y/x)$                    &  $\mathbf{r}_1(u,v)$  \\
$(x^2+y^2-2)\mathrm{Arctan}(y/x)/2\sqrt{2}$          &  $\tilde{\mathbf{r}}_1(u,v)$ \\
$-x\mathrm{Arctan}(y/x)$                             &  $\mathbf{r}_2(u,v)$    \\
$(x\cos\theta +y\sin\theta )^2(1-1/(x^2+y^2))/2$   &  $\mathbf{r}^{\theta }_3(u,v)$ \\
$(x\cos\vartheta+y\sin\vartheta)^2(1-2/(x^2+y^2))/4\sqrt{2}$         &  $\tilde{\mathbf{r}}^{\vartheta }_3(u,v)$ \\
$a(x^2+y^2)+bx+cy+d$                                    & oriented sphere\\
\hline
\end{tabular}
\end{table}

A Laguerre minimal surface enveloped by an elliptic family of cones
is obtained from the surface~(\ref{eq-th2-2}) by transformation from Proposition~\ref{cor1}.
Let us give some typical examples obtained from graphs of the functions in the left column of Table~\ref{trans-elliptic}; see also
Figure \ref{figure-elliptic}. These examples are ``building blocks'' whose convolutions form all the surfaces in question. We represent them in special parametric form $\mathbf{r}(u,v)$, where the map $\mathbf{r}(u,v)$ is the inverse of the composition of the Gaussian spherical map and the stereographic projection. This is convenient to get easy expressions for the convolution surfaces. The choice of building blocks is a question of taste; we choose them to get the simplest possible expressions for $\mathbf{r}(u,v)$.

\begin{example} \label{ex1} The first
building block is the well-known helicoid which
is given implicitly by $x=-y\tan(z/2)$. It can
be parametrized via
\begin{multline}\label{eq1}
\mathbf{r}_1(u,v)=\left( \, u-\frac{u}{u^2+v^2},\,\frac{v}{u^2+v^2}-v,\, 2 \mathrm{Arctan}\,\frac{u}{v}\, \right) \qquad\text{or as a ruled surface via}\\[-7pt]
\end{multline}
\begin{multline}\label{eq1-kinem}
\mathbf{R}_1(\varphi,\lambda)=\left( \, 0,\,0,\, -2\varphi\, \right)
+ \lambda\left( \, \sin\varphi,\,\cos\varphi,\, 0 \, \right).\\[-7pt]
\end{multline}
\end{example}

\begin{example} \label{ex2} The next example
is the cycloid $\mathbf{r}(t)=(t-\sin t,1-\cos t,0)/2$.
One should think of a \emph{cycloid} as a Legendre surface formed by all the contact elements $(r,P)$
such that the plane $P$ passes through the line tangent to the curve $\mathbf{r}(t)$ at the point $r$ of the curve; see the definitions in \S\ref{ssec:laguerre}. We use the parametrization:
\begin{multline}\label{eq2}
\mathbf{r}_2(u,v)=\left(\,\mathrm{Arctan}\,\frac{u}{v}-\frac{u v}{u^2+v^2},\, \frac{u^2}{u^2+v^2},\, 0\,\right).\\[-7pt]
\end{multline}
The family of tangent lines to the cycloid can be parametrized via
\begin{multline}\label{eq2-kinem}
\mathbf{R}_2(\varphi,\lambda)=\left( \, \varphi,\,1,\, 0 \, \right)
+ \lambda\left( \, \sin\varphi,\,\cos\varphi,\, 0 \, \right).\\[-7pt]
\end{multline}
\end{example}

\begin{example} \label{ex3} The third building block is the Pl\"ucker conoid
$z=y^2/(x^2+y^2)$. In parametric form it can be written as:
\begin{multline}\label{eq3}
\mathbf{r}_3(u,v)=\frac{uv}{u^2+v^2}\left(\,
\frac{v}{u^2 + v^2}-v,\, u-\frac{u}{u^2 + v^2},\, \frac{u}{v}
\, \right)\qquad\text{or as}\\[-7pt]
\end{multline}
\begin{multline}\label{eq3-kinem}
\mathbf{R}_3(\varphi,\lambda)=\left( \, 0,\,0,\, \cos2\varphi+1\, \right)\!/2
\,+\, \lambda\left( \, \sin\varphi,\,\cos\varphi,\, 0 \, \right).\\[-7pt]
\end{multline}
The Pl\"ucker conoid has the special property that it arises as an $L$-minimal ruled surface
and at the same time as an i-Willmore surface carrying a $2$-parametric family of i-M-circles.
\end{example}


\begin{figure}[htbp]
\centering
\begin{overpic}[width=\textwidth]{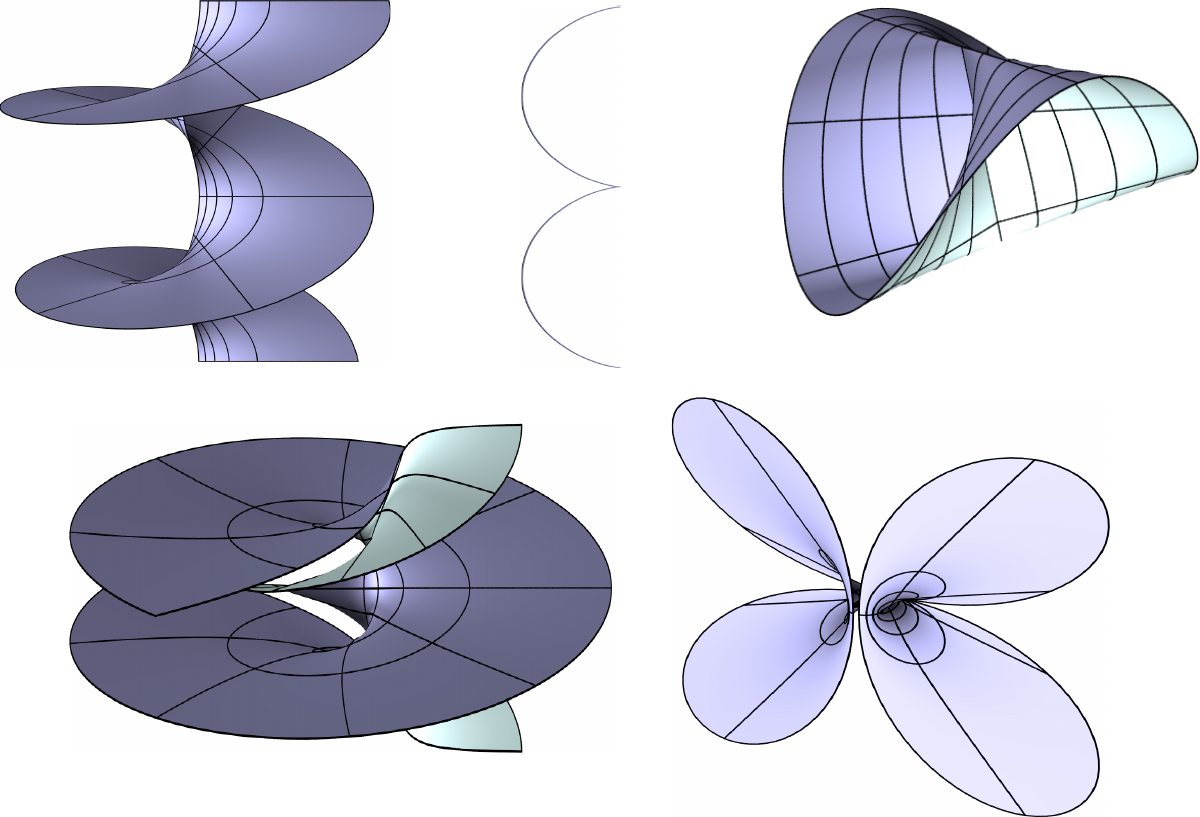}
 \put(10,39){$\mathbf{r}_1$}
 \put(50,42){$\mathbf{r}_2$}
 \put(80,39){$\mathbf{r}_3$}
 \put(10,5){$\tilde{\mathbf{r}}_1$}
  \put(70,5){$\tilde{\mathbf{r}}_3$}
\end{overpic}
 \caption{Building blocks for L-minimal surfaces enveloped by an elliptic family of cones. Starting from the top left we show the surfaces $\mathbf{r}_1$, $\mathbf{r}_2$, $\mathbf{r}_3$,
 $\tilde{ \mathbf{r}}_3$ and $\tilde{ \mathbf{r}}_1$ in clockwise direction.
 Note that the cycloid $\mathbf{r}_2$ lies in a plane orthogonal to the $z$-axis.
 For details refer to Examples \ref{ex1}, \ref{ex2}, \ref{ex3} and Section \ref{ssec:laguerre}.
 }
\label{figure-elliptic}
\end{figure}
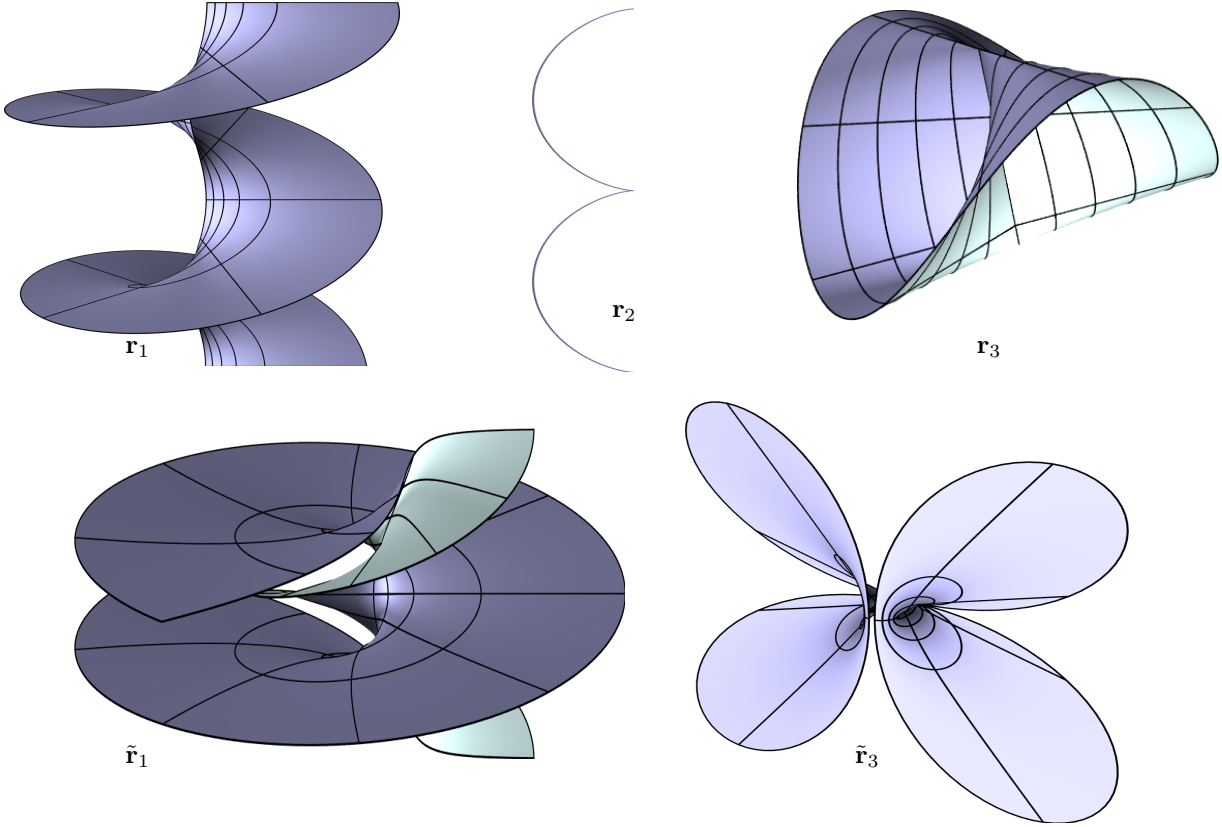

We shall now see that an arbitrary ruled L-minimal surface is up to isometry a convolution of these building blocks; see Figure~\ref{figure-ruled} to the bottom. Denote by
$\mathbf{r}^\theta(u,v)=R^\theta\mathbf{r}(R^{-\theta}(u,v))$,
where $R^\theta$ is the counterclockwise rotation through an angle $\theta$ around the $z$-axis; in this formula the plane $(u,v)$ is identified with the plane $z=0$.


\begin{theorem}[Classification of ruled L-minimal surfaces]\label{th-ruled} A ruled Laguerre minimal surface is up to isometry a piece of the surface
\begin{multline}\label{eq-th-ruled}
\mathbf{r}(u,v)=a_1\mathbf{r}_1(u,v)+a_2\mathbf{r}_2(u,v)+a_3\mathbf{r}^\theta_3(u,v)\\[-7pt]
\end{multline}
for some $a_1,a_2,a_3,\theta\in\mathbb{R}$ such that $a_1^2+a_3^2\ne0$.
Conversely, any immersed piece of the surface~\textup{(\ref{eq-th-ruled})} is ruled and Laguerre minimal.
\end{theorem}


\begin{proof}[of Theorem~\ref{th-ruled}] Let us prove the direct implication.
Let $\Phi$ be a ruled L-minimal surface. By Corollary~\ref{cor-catalan} it follows that $\Phi$ contains a family of line segments parallel to one plane. Choose a coordinate system so that the plane is $Oxy$.

Consider the corresponding surface $\Phi^i$ in the isotropic model. By Theorem~\ref{pgm} it follows that
(a piece of) the surface $\Phi^i$ is a graph of a function $F$ biharmonic in a region $U\subset\mathbb{R}^2$. Since the surface $\Phi$ carries a family of lines parallel to the plane $Oxy$ it follows that the surface $\Phi^i$ carries a family of i-M-circles of the form~(\ref{line}) with $m_3=n_3$. Thus the restriction of the function $F$ to the intersection of the region $U$ with each line $y=tx$, where $t$ runs through a segment $J\subset\mathbb{R}$, is a quadratic function $m_3(t)(x^2+y^2-1)-m_1(t)x-m_2(t)y$.

By Proposition~\ref{auxillary} it follows that formula~(\ref{eq-auxillary}) holds. In this formula $a_1+a_3=b_1+c_1=b_2+c_2=b_3+c_3=0$ because the restriction of the function $F$ to the lines $y=tx$ has special
form $m_3(t)(x^2+y^2-1)-m_1(t)x-m_2(t)y$.

Let us simplify expression~(\ref{eq-auxillary}) by appropriate isometries of $\mathbb{R}^3$
(corresponding to i-M-transformations of the isotropic model; see Table~\ref{isotropic-trans}).
First perform an appropriate rotation of $\mathbb{R}^3$ around the $z$-axis to achieve $a_4=0$ in formula~(\ref{eq-auxillary}) and appropriate translations along the $x$- and $y$-axes to achieve $d_1=d_2=0$.
Bringing to the diagonal form one gets $c_1x^2+c_2xy+c_3y^2=a\left(x\sin\theta+y\cos\theta\right)^2+c\left(x^2+y^2\right)$ for some numbers $a,\theta,c\in\mathbb{R}$. Perform the translation by vector $(0,0,-c)$
along the $z$-axis.

After all the above isometries the function~(\ref{auxillary}) becomes a linear combination in the first, third and fourth functions in the left column of Table~\ref{trans-elliptic}.
By Proposition~\ref{cor1} transformation~(\ref{eq-cor1}) takes the functions in the left column of Table~\ref{trans-elliptic} to the surfaces in the right column. Since the expression in the right-hand side of the formula~(\ref{eq-cor1}) is linear in $F$ the direct implication in the theorem follows. We exclude the case $a_1=a_3=0$ because it does not lead to an immersion.

\begin{proposition} \label{rulings} If $a_1^2+a_3^2\ne 0$ then the surface~\textup{(\ref{eq-th-ruled})}
contains the family of lines
\begin{multline}\label{eq-proof-th-kineruled}
\mathbf{R}(\varphi,\lambda)=a_1\mathbf{R}_1(\varphi,\lambda)+a_2\mathbf{R}_2(\varphi,\lambda)+a_3R^\theta\mathbf{R}_3(\varphi-\theta,\lambda),\\[-7pt]
\end{multline}
where $\varphi$ is the family parameter and $\lambda$ is the line parameter.
\end{proposition}

\begin{proof}[of Proposition~\ref{rulings}]
Fix a number $\varphi\in\mathbb{R}$. Consider three parallel lines
$\mathbf{R}_1(\varphi,\lambda)$, $\mathbf{R}_2(\varphi,\lambda)$, and $R^\theta\mathbf{R}_3(\varphi-\theta,\lambda)$.
Then the line $\mathbf{R}(\varphi,\lambda)$ given by~(\ref{eq-proof-th-kineruled}) is their convolution as Legendre surfaces (to a line $L$ we assign the Legendre surface $\{\,(r,P)\in ST\mathbb{R}^3\,:\,r\in L, P\supset L\,\}$).
Since the lines $\mathbf{R}_1(\varphi,\lambda)$ and $\mathbf{R}_3(\varphi,\lambda)$ are contained in the surfaces $\mathbf{r}_1(u,v)$ and $\mathbf{r}_3(u,v)$, respectively, and the line $\mathbf{R}_2(\varphi,\lambda)$ is tangent to the curve $\mathbf{r}_2(u,v)$
it follows that the line $\mathbf{R}(\varphi,\lambda)$ is contained in the convolution surface $\mathbf{r}(u,v)$ unless $a_1=a_3=0$.
\qed
\end{proof}

Now complete the proof of Theorem~\ref{th-ruled} by checking its reciprocal implication. By Proposition~\ref{rulings} it follows that the surface~(\ref{eq-th-ruled}) is ruled unless $a_1=a_3=0$ (when neither piece of the surface is immersed). By Proposition~\ref{cor1}, Theorem~\ref{pgm}, and Table~\ref{trans-elliptic} it follows that any immersed piece of the surface~(\ref{eq-th-ruled}) is L-minimal.
\qed\end{proof}


\begin{proof}[of Theorem~\ref{th-kineruled}]
By Theorem~\ref{th-ruled} a ruled L-minimal surface can up to isometry be parametrized via~(\ref{eq-th-ruled})  with $a_1^2+a_3^2\ne0$. By Proposition~\ref{rulings} the surface can also be parametrized via~(\ref{eq-proof-th-kineruled}).
It remains to notice that up to isometry formulas~(\ref{eq-proof-th-kineruled}) and~(\ref{eq-ruled-intro}) define the same class of surfaces.
\qed\end{proof}


\begin{figure}[htbp]
\centering
\includegraphics[width=\textwidth]{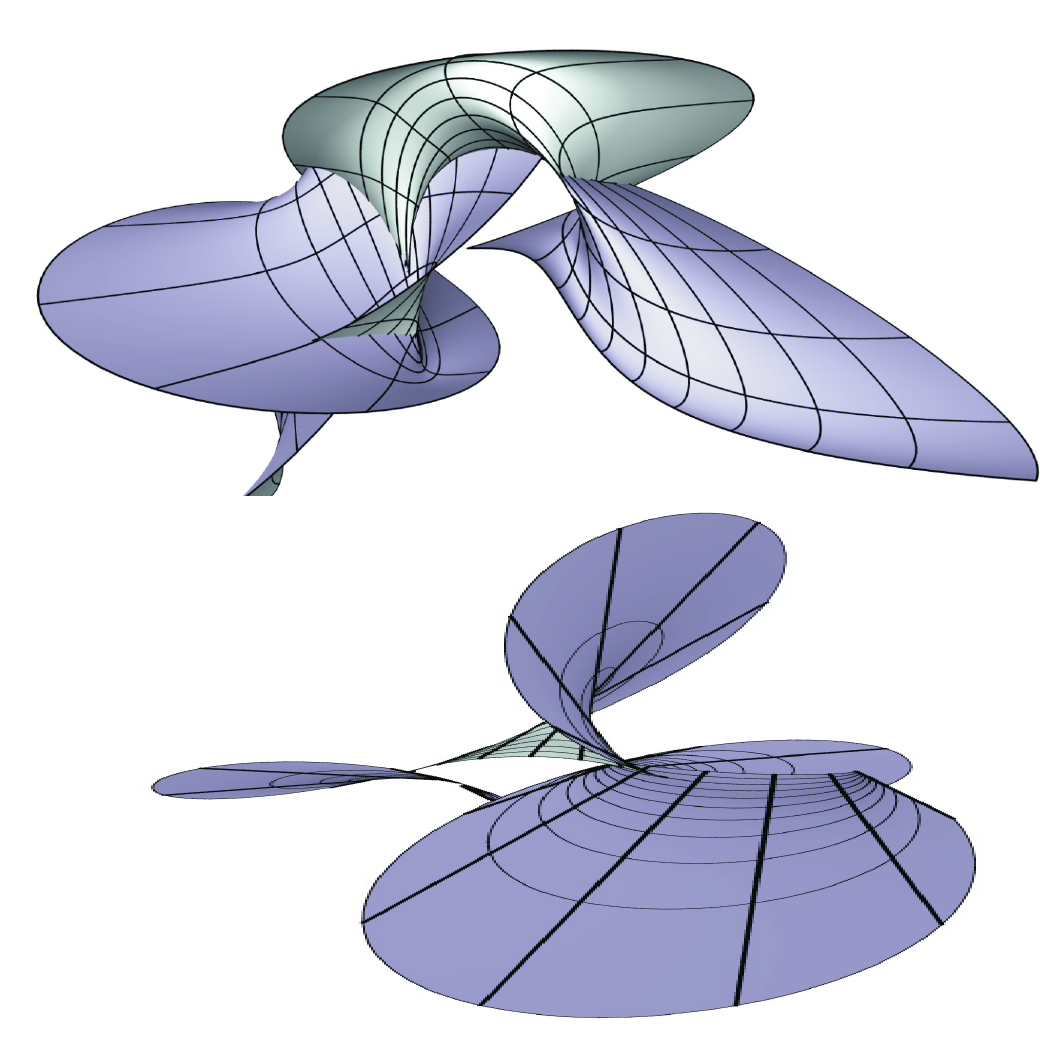}
 \caption{ (Top) A general L-minimal surface enveloped by
an elliptic family of cones. For more information refer to Definition~\ref{def-elliptic}
and Corollary~\ref{classification-elliptic}. (Bottom) A general ruled L-minimal surface is a convolution surface of the surfaces $\mathbf{r}_1$, $\mathbf{r}_2$ and $\mathbf{r}_3$. The rulings are depicted in black. For more information refer to Theorem~\ref{th-ruled}.}
 \label{figure-ruled}
\end{figure}

An L-minimal surface enveloped by an elliptic family of cones can be obtained from Examples~\ref{ex1}--\ref{ex3} by performing L-transformations and taking special convolution surfaces (in general convolution operation does not preserve the class of surfaces enveloped by a family of cones); see Figure~\ref{figure-ruled} to the top. Recall that $\tilde{\mathbf{r}}(u,v)$ is the surface obtained from a surface $\mathbf{r}(u,v)$ by the L-transformation~$\Lambda$; see \S\ref{ssec:laguerre} for the definition.


\begin{corollary}[Classification for elliptic type]\label{classification-elliptic}
A Laguerre minimal surface
enveloped by an elliptic family of cones 
is Laguerre equivalent to a piece of the surface
\begin{multline}\label{eq-classification-elliptic}
\mathbf{r}(u,v)=a_{1}\mathbf{r}_1(u,v)+a_2\mathbf{r}_2(u,v)
+a_3\mathbf{r}^{\theta}_3(u,v)+
a_4\tilde{\mathbf{r}}_1(u,v)+
a_5\tilde{\mathbf{r}}^{\vartheta}_3(u,v)\\[-10pt]
\end{multline}
for some $a_1,a_2,a_3,a_4,a_5,\theta ,\vartheta \in\mathbb{R}$.
Conversely, an immersed piece of surface~\textup{(\ref{eq-classification-elliptic})} is Laguerre minimal and is enveloped by an elliptic family of cones.
\end{corollary}

\begin{proof}[of Corollary~\ref{classification-elliptic}]
Let us prove the direct implication.
Let $\Phi$ be a Laguerre minimal surface enveloped by an elliptic family of cones.
Then the surface $\Phi^i$ carries a family of i-M-circles such that the top view of the family is an elliptic pencil. By Theorem~\ref{i-classification-elliptic}
it follows that the surface $\Phi^i$ is i-M-equivalent to surface~(\ref{eq-th2-2}).

The right-hand side of formula~(\ref{eq-th2-2}) is a linear combination in the expressions in the left column of Table~\ref{trans-elliptic}.
Performing an i-M-transformation $z\mapsto z+a(x^2+y^2)+bx+cy+d$ one can eliminate the last expression from the linear combination.
By Proposition~\ref{cor1} and Table~\ref{isotropic-trans} transformation $\Phi^i\mapsto \Phi$ takes the functions in the left column of Table~\ref{trans-elliptic} to the surfaces in the right column. Since the expression in the right-hand side of formula~(\ref{eq-cor1}) is linear in $F$ the direct implication follows.

The converse implication follows from Propositions~\ref{cor1},~\ref{observation}, Theorem~\ref{pgm}, and Table~\ref{trans-elliptic}.
\qed\end{proof}

\bigskip

\textbf{Description of the families of cones.}
Let us describe the families of cones which make up the L-minimal surfaces in question. We view a cone as a linear family of oriented spheres.
If we map an oriented sphere with midpoint $(m_1,m_2,m_3)$ and signed radius $R$ to the point
$(m_1,m_2,m_3,R)\in \mathbb{R}^4$, we get a correspondence between cones in $\mathbb{R}^3$
and lines in $\mathbb{R}^4$. Surfaces enveloped by a family of cones can be regarded
as ruled $2$-surfaces in $\mathbb{R}^4$.
Laguerre transformations of $\mathbb{R}^3$ correspond to Lorentz transformations of $\mathbb{R}^4$ under
this mapping. This is known as the \emph{cyclographic model} of Laguerre geometry; see
\cite{pottmann-1998-alcagd} for more information.
We will refer to a ruled $2$-surface in $\mathbb{R}^4$ corresponding to
a surface enveloped by a family of cones as a \emph{cyclographic preimage}.

\begin{proposition}\label{ex-tildas}
	The cyclographic preimage 
        of the surface $\tilde{\mathbf{r}}_1(u,v)$ can be parametrized as
\begin{multline}\label{eq1-tilda-cycl} \tilde{\mathbf{R}}_1(\varphi,\lambda)=\left(0,\,0,\,-3\varphi,\,\varphi\right)\!/2\sqrt{2}
        \,+\,\lambda\left(\sin\varphi,\,\cos\varphi,\,0,\,0\right).\\[-7pt]
\end{multline}
The cyclographic preimage of the surface $\tilde{\mathbf{r}}_3(u,v)$ can be parametrized as
\begin{multline} \tilde{\mathbf{R}}_3(\varphi,\lambda)=\left(0,\,0,\,3\cos^2\varphi,-\cos^2\varphi\right)\!/4\sqrt{2}
	\,+\, \lambda\left(\sin\varphi,\,\cos\varphi,\,0,\,0\right).\\[-7pt]
\end{multline}
\end{proposition}

\begin{proof}[of Proposition~\ref{ex-tildas}]
Let us find the cyclographic preimage of the surface $\tilde{\mathbf{r}}_1(u,v)$. Consider the image $\Phi^i$ of the surface in the isotropic model. By Table~\ref{trans-elliptic} the image $\Phi^i$ has the equation $z=(x^2+y^2-2)\mathrm{Arctan}(y/x)/2\sqrt{2}$. Thus for each $\varphi\in\mathbb{R}$ the surface $\Phi^i$ contains the i-circle
\begin{equation*}
\begin{cases}
z=-(x^2+y^2-2)\varphi/2\sqrt{2},\\
0=x\sin\varphi+y\cos\varphi.
\end{cases}
\end{equation*}
By formula~(\ref{iso-sphere}) this i-circle is the image (in the isotropic model) of the limit of the sequence of the cones touching the two spheres with cyclographic coordinates $\left(0,\,0,\,-3\varphi,\,\varphi\right)/2\sqrt{2}$
and $N\left(\sin\varphi,\,\cos\varphi,\,0,\,0\right)$, where $N\to\infty$. Thus the cyclographic preimage of the surface $\tilde{\mathbf{r}}_1(u,v)$ is given by the required formula~(\ref{eq1-tilda-cycl}). The cyclographic preimages of the surface $\tilde{\mathbf{r}}_3(u,v)$ and all the other surfaces below are computed analogously.
\qed\end{proof}

\begin{theorem}\label{cyclo-elliptic}
	The cyclographic preimage of a Laguerre minimal surface enveloped by
	an elliptic family of cones is up to Lorentz transformations a piece of the surface
	\begin{multline}\label{eq-th-elliptic-cones}
		\mathbf{R}(\varphi,\lambda)=
		\left(\, A\varphi ,\, B\varphi,\,
		C\varphi + D\cos2\varphi,\,
		E\varphi + F\cos2\varphi + G\sin2\varphi \, \right)
		+\lambda \left(\, \sin\varphi,\, \cos\varphi,\, 0,\, 0\, \right)\\[-7pt]
	\end{multline}
	for some $A,B,C,D,E,F,G\in\mathbb{R}$.
\end{theorem}

In other words, an L-minimal surface enveloped by an elliptic family of cones can be interpreted as a frequency $1$ rotation of a line in a plane, plus a frequency $2$ ``harmonic oscillation'', and a constant-speed translation; this time in $\mathbb{R}^4$.

\begin{proof}
Notice that the cyclographic preimage of a ruled surface in $\mathbb{R}^3$ is the surface itself,
if $\mathbb{R}^3$ is identified with subspace $R=0$ of $\mathbb{R}^4$.
Analogously to the proof of Proposition~\ref{rulings} one can show that the cyclographic preimage
of surface~(\ref{eq-classification-elliptic}) can be parametrized via
\begin{multline}\label{eq-proof-cyclo}
\mathbf{R}(\varphi,\lambda)=a_1\mathbf{R}_1(\varphi,\lambda)+a_2\mathbf{R}_2(\varphi,\lambda)
+a_3R^{\theta}\mathbf{R}_3(\varphi-\theta,\lambda)+a_4\tilde{\mathbf{R}}_1(\varphi,\lambda)
+a_5R^{\vartheta}\tilde{\mathbf{R}}_3(\varphi-\vartheta,\lambda),\\[-7pt]
\end{multline}
with the same $a_1,a_2,a_3,a_4,a_5,\theta ,\vartheta \in\mathbb{R}$.
By Examples~\ref{ex1}--\ref{ex3} and Proposition~\ref{ex-tildas} such a parametrization gives the same class of surfaces as the required parametrization~(\ref{eq-th-elliptic-cones}).
\qed\end{proof}



\subsection{Hyperbolic families of cones}
Here we consider the second kind of L-minimal surfaces enveloped by a family of
cones.

\begin{definition} \label{def-hyperbolic} A \emph{hyperbolic} pencil of circles in the plane (or in a sphere) is the set of all the circles
orthogonal to two fixed crossing circles.
A $1$-parametric family of cones (possibly degenerating to cylinders or lines) in space is
\emph{hyperbolic} if the Gaussian spherical images of the cones form a hyperbolic pencil of circles in the unit sphere.
\end{definition}

\begin{theorem} \label{i-classification-hyperbolic}
Let $\Phi^i$ be an i-Willmore surface carrying a family of i-M-circles.
Suppose that the top view of the family is a hyperbolic pencil of circles.
Then the surface $\Phi^i$ is i-M-equivalent to a piece of the surface
\begin{multline}
z=\left(a_1(x^2+y^2)+a_2x+a_3\right)\ln (x^2+y^2)+
\frac{b_1y+b_2x}{x^2+y^2}+(c_1y+c_2x)(x^2+y^2)\label{eq-th2-3}\\[-10pt]
\end{multline}
for some $a_1,a_2,a_3,b_1,b_2,c_1,c_2\in\mathbb{R}$.
\end{theorem}

\begin{proof}[of Theorem~\ref{i-classification-hyperbolic}]
Perform an i-M-transformation taking the hyperbolic pencil of circles in the top view into the pencil of concentric circles $x^2+y^2=t$, where $t$ runs through a segment $J\subset\mathbb{R}$.
Denote by $z=F(x,y)$ the surface obtained from the surface $\Phi^i$ by the transformation, where $F$ is a biharmonic function defined in a region $U\subset\mathbb{R}^2$.
Since an i-M-transformation takes i-M-circles to i-M-circles it follows that
the restriction of the function $F$ to (an appropriate arc of) each circle $x^2+y^2=t$, where $t\in J$, is a linear function.

Without loss of generality assume that $(0,0)\not\in U$. Consider the polar coordinates in $U$.
Then $F(r,\phi)=a(r)\cos\phi+b(r)\sin\phi+c(r)$.
Thus
\begin{multline*}
r^4\Delta^2F=r^4F_{rrrr}+2r^3F_{rrr}-r^2F_{rr}+rF_r+2r^2F_{rr\phi\phi}-2rF_{r\phi\phi}+4F_{\phi\phi}+F_{\phi\phi\phi\phi}=\\
\begin{aligned}
&=\left(r^{4}a^{(4)}+2r^3a^{(3)}-3r^2a''+3ra'-3a\right)\cos\phi+\\
&+\left(r^{4}b^{(4)}+2r^3b^{(3)}-3r^2b''+3rb'-3b\right)\sin\phi+\\
&+\left(r^{4}c^{(4)}+2r^3c^{(3)}-r^2c''+rc'\right).
\end{aligned}
\end{multline*}
Since $\Delta^2F=0$ it follows that the coefficients of this trigonometric polynomial vanish.
Solving the obtained ordinary differential equations we get:
\begin{align*}
a(r)&=\alpha_1r+\alpha_2r\ln r+\alpha_3/{r}+\alpha_4r^3;\\
b(r)&=\beta_1r+\beta_2r\ln r+\beta_3/{r}+\beta_4r^3;\\
c(r)&=\gamma_1+\gamma_2r^2+\gamma_3\ln r+\gamma_4r^2\ln r.
\end{align*}
One can achieve $\beta_2=0$ by an appropriate rotation of the coordinate system around the origin.
One can also achieve $\alpha_1=\beta_1=\gamma_1=\gamma_2=0$ by
the i-M-transformation $z\mapsto z-\gamma_2(x^2+y^2)-\alpha_1x-\beta_1y-\gamma_1$
Returning to the \mscomm{Cartesian} coordinate system we get the 
required formula. 
\qed\end{proof}

\begin{table}\caption{Biharmonic functions whose restrictions to each circle $x^2+y^2=t$, $t\in I$, are linear functions and corresponding Laguerre minimal surfaces}
\label{trans-hyperbolic}
\begin{tabular}{|l|r|}
\hline
Biharmonic function & Laguerre minimal surface\\
\hline
$(x^2+y^2-1)(\ln(x^2+y^2)-2)/2-2$                    &  $\mathbf{r}_4(u,v)$ \\
$(x^2+y^2-2)(\ln(x^2+y^2)-2-\ln2)/4\sqrt{2}-\sqrt{2}$&  $\tilde{\mathbf{r}}_4(u,v)$\\
$x\ln(x^2+y^2)+x(x^2+y^2-1)$                         &  $\mathbf{r}_5(u,v)$ \\
$(x\cos\theta +y\sin\theta )(x^2+y^2-2+1/(x^2+y^2))$&  $\mathbf{r}^{\theta}_6(u,v)$          \\
$(x\cos\vartheta +y\sin\vartheta )(x^2+y^2-4+4/(x^2+y^2))/4$&
$\tilde{\mathbf{r}}^{\vartheta}_6(u,v)$      \\
$a(x^2+y^2)+bx+cy+d$                                    & oriented sphere\\
\hline
\end{tabular}
\end{table}

An L-minimal surface enveloped by a hyperbolic family of cones
is obtained from the surface~(\ref{eq-th2-3}) by ``transformation''~(\ref{eq-cor1}).
Let us give some typical examples obtained from the graphs of the functions in the left column of Table~\ref{trans-hyperbolic}; see also Figure \ref{figure-hyperbolic}.
These examples are building blocks forming all the surfaces in question.

\begin{example} \label{ex4}The first example is the catenoid.
It can be parametrized as
\begin{multline}\label{eq4}
\mathbf{r}_4(u,v)=\left(\, u + \frac{u}{u^2 + v^2},\, v + \frac{v}{u^2 + v^2},\, \ln(u^2 + v^2)\, \right).\\[-7pt]
\end{multline}
Its cyclographic preimage can be written as
\begin{multline}
\mathbf{R}_4(\varphi,\lambda) =
\left(\, 0,\, 0,\, -2\varphi ,\, -2\, \right) + \lambda \left(\, 0,\, 0,\, \cosh\varphi,\,
\sinh\varphi\, \right).\\[-7pt]
\end{multline}
\end{example}

\begin{example} \label{ex5} Another building block is given by the surface $\mathbf{r}_5$
parametrized by:
\begin{multline}\label{eq5}
\mathbf{r}_5(u,v)=\left(\, (u^2 - v^2)\left(1 - \frac{1}{u^2 + v^2}\right) - \ln(u^2 + v^2),\, 2 u v \left(1 - \frac{1}{u^2 + v^2}\right),\, 4 u\,\right).\\[-7pt]
\end{multline}
Its cyclographic preimage is the surface parametrized by
\begin{multline}
\mathbf{R}_5(\varphi,\lambda) =
\left(\, 1 - e^{-2\varphi}+2\varphi,\, 0,\, 0,
\, 0\, \right) \, +\, \lambda \left(\, 0,\, 0,\, \cosh\varphi,\,
\sinh\varphi\, \right).\\[-7pt]
\end{multline}
\end{example}

\begin{example} \label{ex6} Finally we have the surface $\mathbf{r}_6$
given implicitly by $z^2(z^2-16x)=64y^2$. In parametric form it can be written as:
\begin{multline}\label{eq6}
\mathbf{r}_6(u,v)=\left(\,
(u^2-v^2) \left(1-\dfrac{1}{u^2 + v^2}\right)^2,\,
2 u v \left(1-\dfrac{1}{u^2 + v^2}\right)^2, \,
4 u \left(1-\dfrac{1}{u^2 + v^2}\right)
\,\right).\\[-7pt]
\end{multline}
This surface is L-minimal and i-Willmore simultaneously. Its cyclographic preimage can be written as
\begin{multline}
\mathbf{R}_6(\varphi,\lambda) =
\left(\, 2 - 2\cosh 2\varphi,\, 0,\, 0,
\, 0\, \right)  +\lambda \left(\, 0,\, 0,\, \cosh\varphi,\,
\sinh\varphi\, \right).\\[-7pt]
\end{multline}
\end{example}

\begin{figure}[htbp]

\begin{overpic}[width=.7\textwidth, angle=-90]{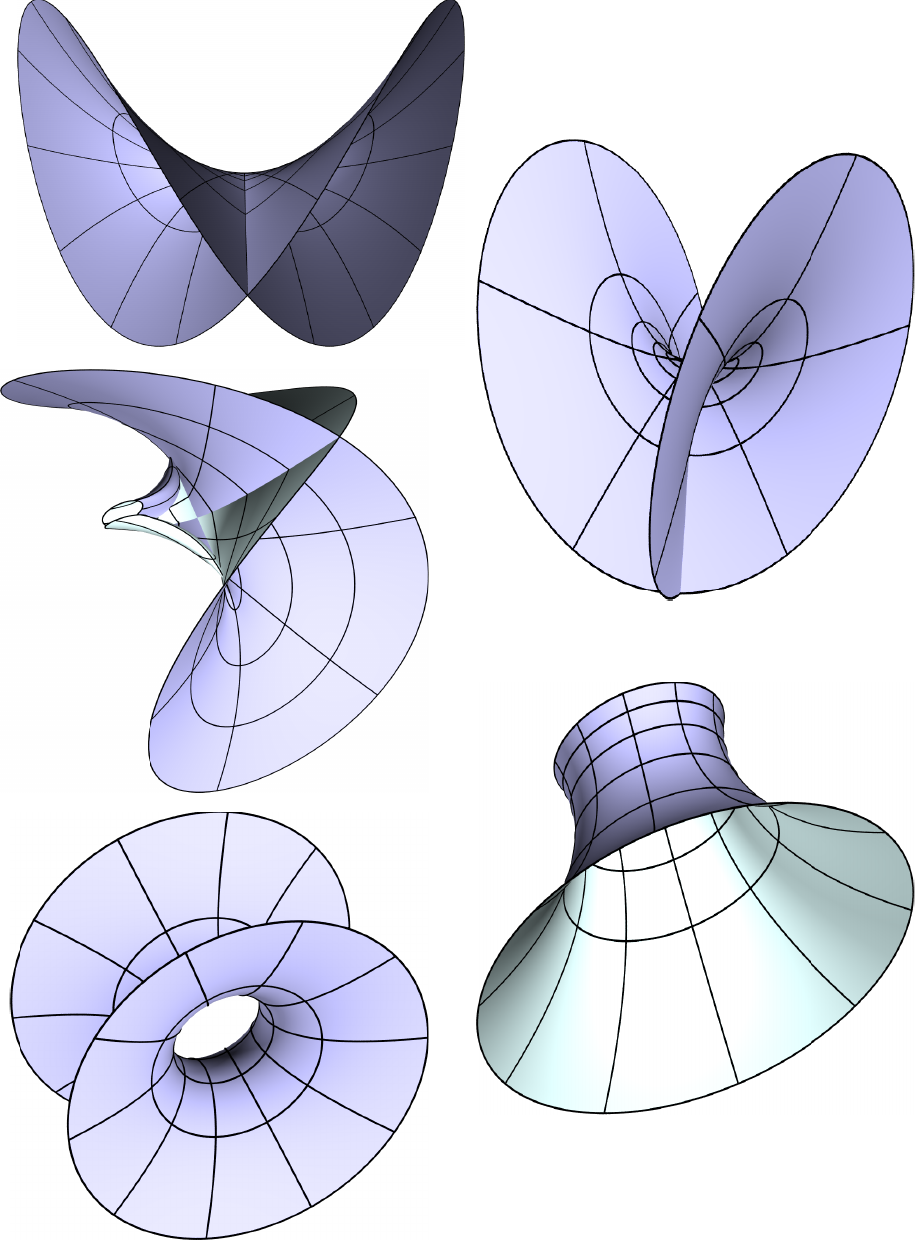}
\put(28,41){$\mathbf{r}_4$}
\put(62,41){$\mathbf{r}_5$}
\put(95,45){$\mathbf{r}_6$}
\put(90,15){$\tilde{\mathbf{r}}_6$}
\put(38,10){$\tilde{\mathbf{r}}_4$}
\end{overpic}
 \caption{Building blocks for L-minimal surfaces enveloped by a hyperbolic family of cones. Starting from the top left we show the surfaces $\mathbf{r}_4$, $\mathbf{r}_5$, $\mathbf{r}_6$,
 $\tilde{\mathbf{r}}_{6}$ and $\tilde{\mathbf{r}}_{4}$ in clockwise direction. For details refer to Examples \ref{ex4}, \ref{ex5}, \ref{ex6} and Section~\ref{ssec:laguerre}.}
 \label{figure-hyperbolic}
\end{figure}
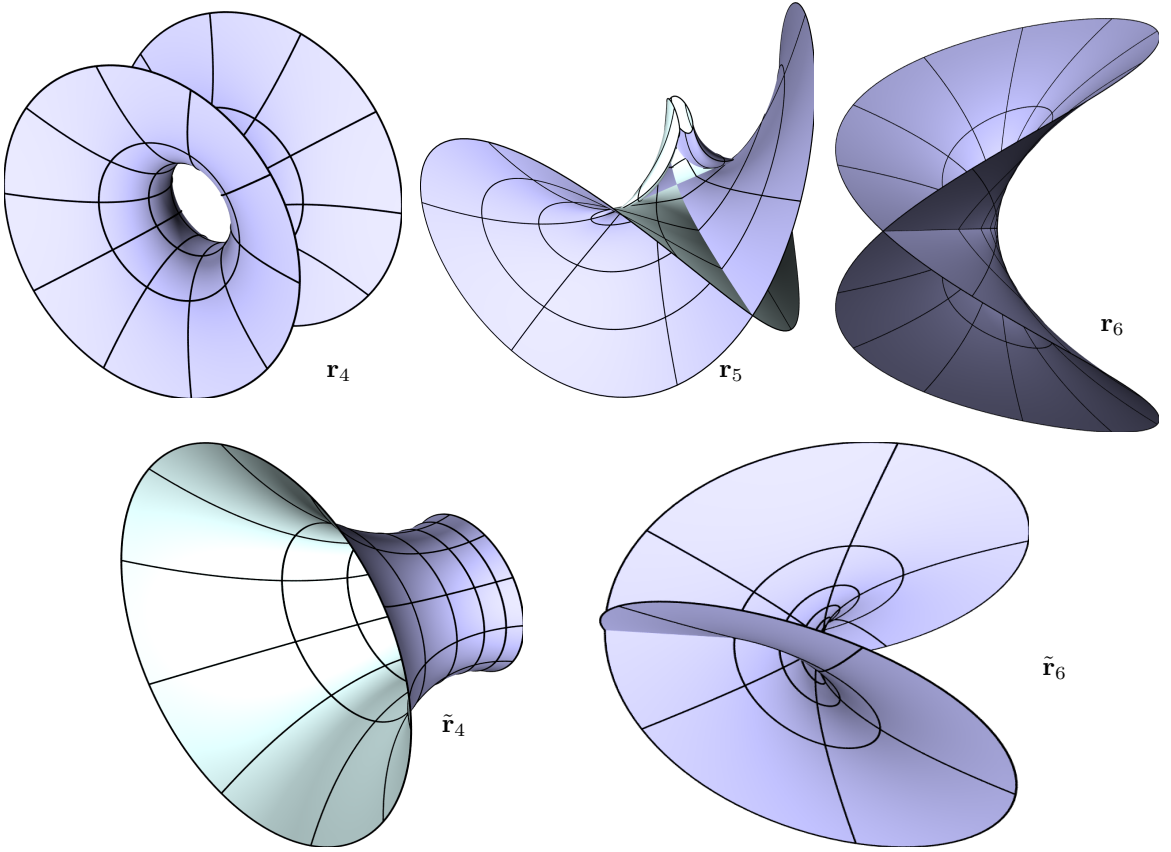



An L-minimal surface enveloped by a hyperbolic family of cones can be obtained from Examples~\ref{ex4}--\ref{ex6} by performing L-transformations and taking special convolution surfaces; see Figure~\ref{figure-hyperbolic-general}:

\begin{corollary}[Classification for hyperbolic type]\label{classification-hyperbolic}
A Laguerre minimal surface
enveloped by a hyperbolic family of cones 
is Laguerre equivalent to a piece of the surface
\begin{multline} \label{eq-classification-hyperbolic}
\mathbf{r}(u,v)=a_1\mathbf{r}_4(u,v)+a_2\mathbf{r}_5(u,v)+
a_3\mathbf{r}^{\theta}_6(u,v)+
a_4\tilde{\mathbf{r}}_4(u,v)+
a_5\tilde{\mathbf{r}}^{\vartheta}_6(u,v)\\[-10pt]
\end{multline}
for some $a_1,a_2,a_3,a_4,a_5,\theta ,\vartheta \in\mathbb{R}$.
Conversely, an immersed piece of surface~\textup{(\ref{eq-classification-hyperbolic})} is Laguerre minimal and is enveloped by a hyperbolic family of cones.
\end{corollary}

\begin{proof}[of Corollary~\ref{classification-hyperbolic}]
Let us prove the direct implication.
Let $\Phi$ be an L-minimal surface enveloped by a hyperbolic family of cones.
Then the surface $\Phi^i$ carries a family of i-M-circles such that the top view of the family is a hyperbolic pencil. By Theorem~\ref{i-classification-hyperbolic}
it follows that the surface $\Phi^i$ is i-M-equivalent to surface~(\ref{eq-th2-3}).

The right-hand side of formula~(\ref{eq-th2-3}) is a linear combination in the expressions in the left column of Table~\ref{trans-hyperbolic}.
Performing an i-M-transformation $z\mapsto z+a(x^2+y^2)+bx+cy+d$ one can eliminate the last expression from the linear combination.
By Proposition~\ref{cor1} and Table~\ref{isotropic-trans} transformation~(\ref{eq-cor1}) takes the functions in the left column of Table~\ref{trans-hyperbolic} to the surfaces in the right column. Since the transformation~(\ref{eq-cor1}) is linear in $F$ the direct implication follows.

The converse implication follows from Propositions~\ref{cor1},~\ref{observation}, Theorem~\ref{pgm}, and Table~\ref{trans-hyperbolic}.
\qed\end{proof}

\bigskip

There is a simple parametrization of the cyclographic preimage
(the proof is analogous to the proof of Theorem~\ref{cyclo-elliptic}).

\begin{theorem}
	The cyclographic preimage of a Laguerre minimal surface enveloped by
	a hyperbolic family of cones is up to Lorentz transformations a piece of the surface
	\begin{multline}\label{eq-th-hyperbolic-cones}
		\mathbf{R}(\varphi,\lambda)=
		\left(\, A\varphi + B\cosh2\varphi,\,
		C\varphi + D\cosh2\varphi + E\sinh2\varphi,\,
		F\varphi ,\, G\varphi \, \right)
		+\lambda \left(\, 0,\, 0,\, \cosh\varphi,\, \sinh\varphi\, \right)\\[-7pt]
	\end{multline}
	for some $A,B,C,D,E,F,G\in\mathbb{R}$.
\end{theorem}

\subsection{Parabolic families of cones}

\begin{definition} \label{def-parabolic} A \emph{parabolic} pencil of circles in the plane (or in a sphere) is the set of all the circles
touching a fixed circle at a fixed point.
A $1$-parametric family of cones (possibly degenerating to cylinders or lines) in space is
\emph{parabolic} if the Gaussian spherical images of the cones form a parabolic pencil of circles in the unit sphere.
\end{definition}

\begin{theorem} \label{i-classification-parabolic}
Let $\Phi^i$ be an i-Willmore surface carrying a family of i-M-circles.
Suppose that the top view of the family is a parabolic pencil of circles.
Then the surface $\Phi^i$ is i-M-equivalent to a piece of the surface
\begin{multline}
z=a_1(5y^2-x^2)x^3+a_2(3y^2-x^2)x^2+
(b_1y^2+b_2xy+b_3x^2)x+c_1y^2+c_2xy\label{eq-th2-1}\\[-10pt]
\end{multline}
for some $a_1,a_2,b_1,b_2,b_3,c_1,c_2\in\mathbb{R}$.
\end{theorem}

\begin{proof}[of Theorem~\ref{i-classification-parabolic}]
Perform an i-M-transformation taking the parabolic pencil of circles in the top view into the pencil of parallel lines $x=t$, where $t$ runs through a segment $J\subset\mathbb{R}$.
Denote by $z=F(x,y)$ the surface obtained from the surface $\Phi^i$ by the transformation, where $F$ is a biharmonic function defined in a region $U\subset\mathbb{R}^2$.
Since an i-M-transformation takes i-M-circles to i-M-circles it follows that the restriction of the function $F$ to (an appropriate segment of) each line $x=t$, where $t\in J$, is a quadratic function.

So $F(x,y)=a(x)y^2+b(x)y+c(x)$. Thus $\Delta^2F=a^{(4)}y^2+b^{(4)}y+c^{(4)}+4a''$. Since $\Delta^2F=0$ it follows that the coefficients of this polynomial in $y$ vanish. Hence
\begin{align*}
a(x)&=\alpha_0+\alpha_1x+\alpha_2x^2+\alpha_3x^3;\\
b(x)&=\beta_0+\beta_1x+\beta_2x^2+\beta_3x^3;\\
c(x)&=\gamma_0+\gamma_1x+\gamma_2x^2+\gamma_3x^3-\alpha_2x^4/3-\alpha_3x^5/5.
\end{align*}
One can achieve $\beta_0=\gamma_0=\gamma_1=\gamma_2=0$ by
the i-M-transformation $z\mapsto z-\gamma_2(x^2+y^2)-\gamma_1x-\beta_0y-\gamma_0$.
We get the 
required formula. 
\qed\end{proof}

\begin{table}[htbp]
\caption{Biharmonic functions whose restrictions to each line $x=t$, $t\in I$, are quadratic functions and corresponding Laguerre minimal surfaces}
\label{trans-parabolic}
\begin{tabular}{|l|r|}
\hline
Biharmonic function & Laguerre minimal surface\\
\hline
$(x\cos\theta  + y\sin\theta )^2/2$                                              &  $\mathbf{r}^{\theta}_7(u,v)$\\
$x^3$                                                &  $\mathbf{r}_8(u,v)$\\
$x^2y$                                               &  $\mathbf{r}_9(u,v)$\\
$xy^2$                                               &  $\mathbf{r}^{\pi/2}_9(u,v)$\\
$x^2(x^2-3 y^2)/2$                                   &  $\mathbf{r}_{10}(u,v)$\\
$x^3(x^2-5 y^2)$                                     &  $\mathbf{r}_{11}(u,v)$\\
$a(x^2+y^2)+bx+cy+d$                                 &  oriented sphere\\
\hline
\end{tabular}
\end{table}

\begin{figure}[htbp]
\centering
\begin{overpic}[width=.8\textwidth , angle = -90]{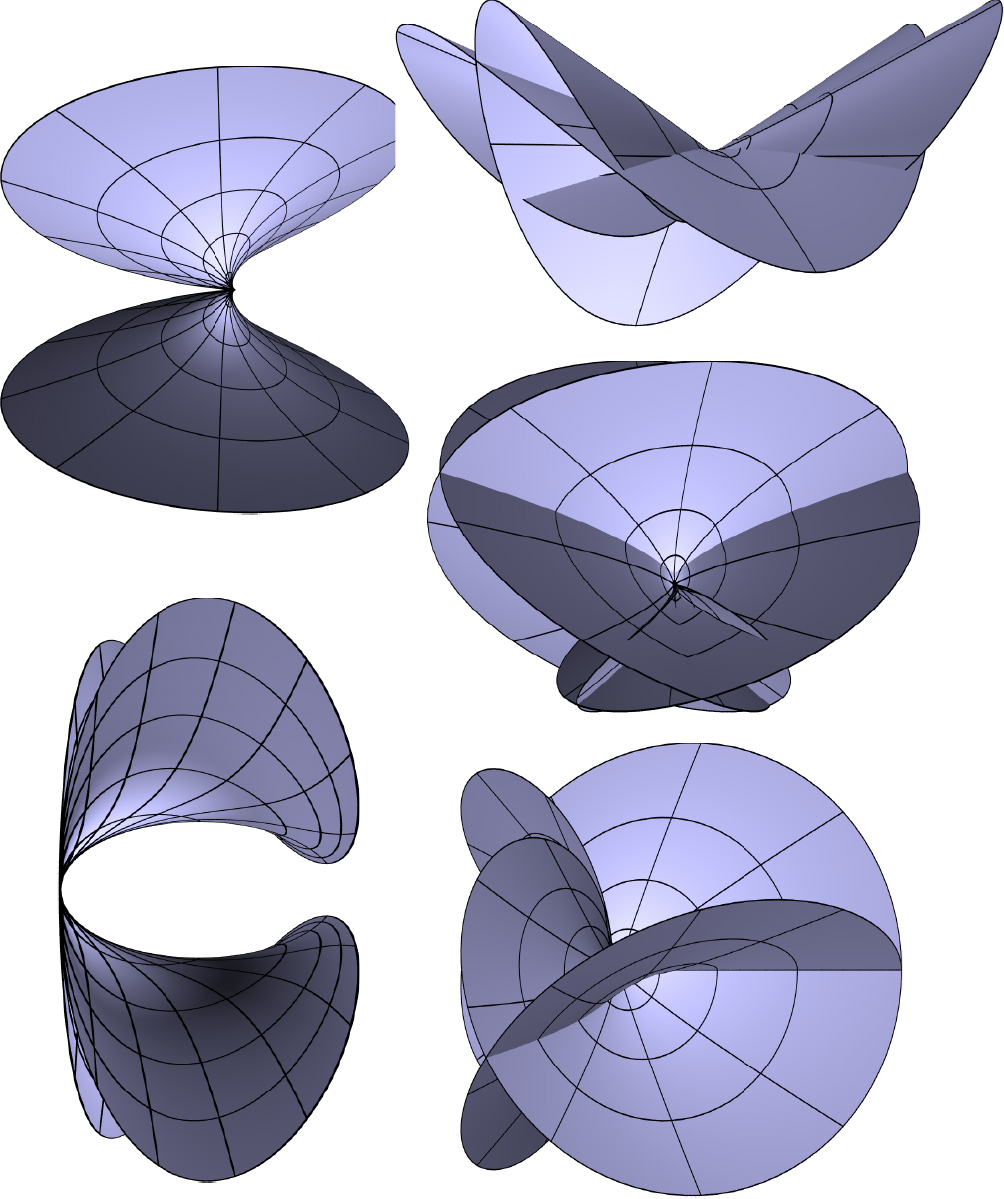}
\put(5,55){$\mathbf{r}_7$}
\put(75,55){$\mathbf{r}_8$}
\put(20,5){$\mathbf{r}_9$}
\put(55,5){$\mathbf{r}_{10}$}
\put(85,5){$\mathbf{r}_{11}$}
\end{overpic}
 \caption{Building blocks for L-minimal surfaces enveloped by a parabolic family of cones. Starting from the top left we show the surfaces $\mathbf{r}_7$, $\mathbf{r}_8$, $\mathbf{r}_{11}$,
 $\mathbf{r}_{10}$ and $\mathbf{r}_{9}$ in clockwise direction. For details refer
 to Examples \ref{ex7}--\ref{ex11}.}
 \label{figure-parabolic}
\end{figure}
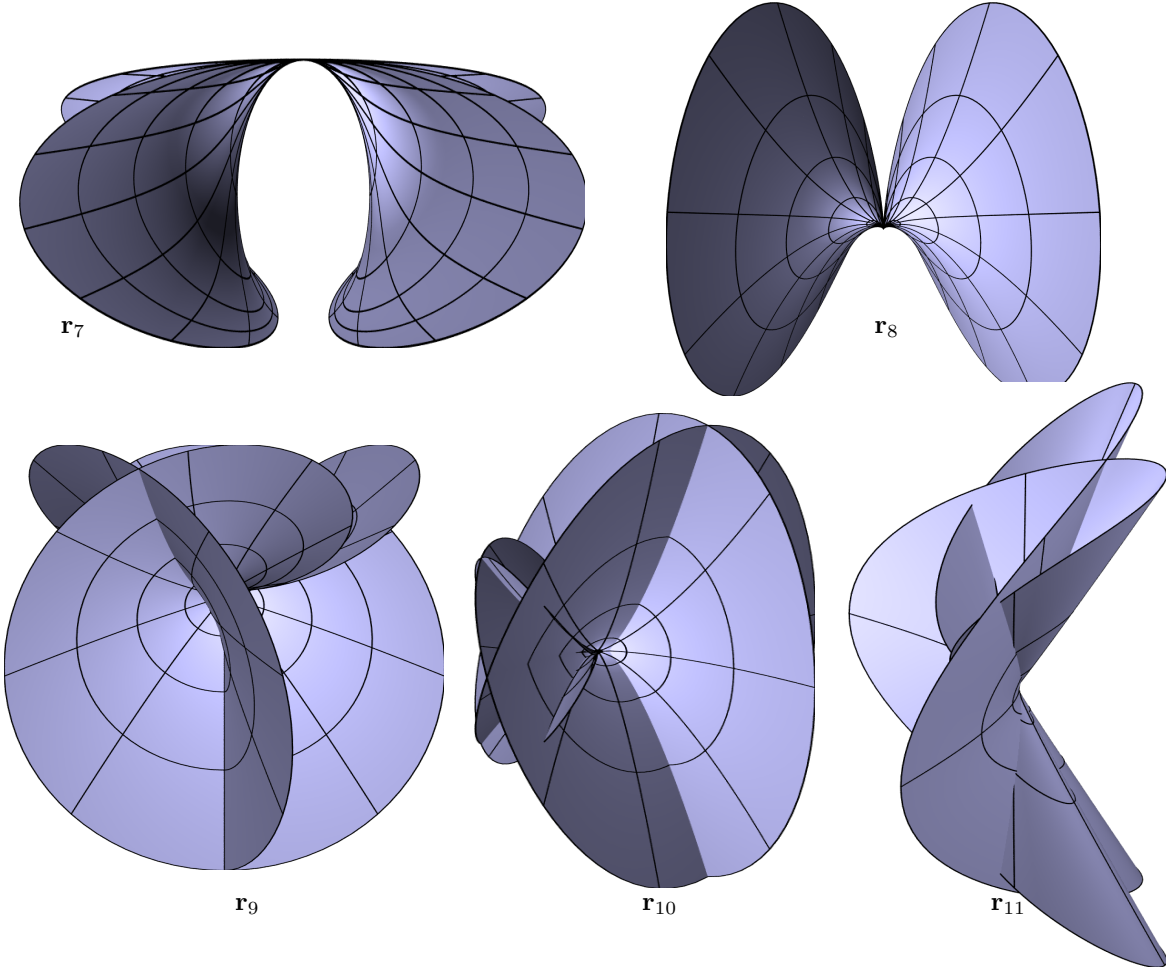

An L-minimal surface enveloped by a parabolic family of cones
is obtained from the surface~(\ref{eq-th2-1}) by transformation~(\ref{eq-cor1}).
Let us give some typical examples obtained from the graphs of the functions in the left column of Table~\ref{trans-parabolic}; see also Figure~\ref{figure-parabolic}.
These examples are building blocks forming all the surfaces in question.

\begin{example} \label{ex7} The first example is the
parabolic horn cyclide $\left(y^2+z^2\right)(1-z)=x^2z$. In parametric form it can be written as:
\begin{multline}\label{eq7}
\mathbf{r}_7(u,v)=\frac{1}{1+u^2+v^2}\left(\,-u-uv^2,\,u^2 v,\,u^2\, \right).\\[-7pt]
\end{multline}
Its cyclographic preimage can be parametrized as
\begin{multline}
\mathbf{R}_7(\varphi,\lambda)
= \left(\, 0,\, 0 ,\, -\varphi^2,\,\varphi^2\, \right)\!/2
\,+\, \lambda \left(1 ,\, 0,\, -\varphi,\, \varphi\, \right) \\[-7pt]
\end{multline}
%
%
%
%
The surface $\mathbf{r}_7$ has the following property: there is a $2$-parametric family of cones touching the cyclide along certain curves. In particular, there are both parabolic and elliptic $1$-parametric families of cones touching the surface along curves.
One of the elliptic families of cones can be parametrized as
\begin{multline}
\left(\, 0,\, 0 ,\, \cos^2\varphi,\, \cos^2\varphi \, \right)\!/2\,+\,\lambda \left(\sin\varphi ,\, \cos\varphi,\, 0,\, 0\, \right)\\[-7pt]
\end{multline}
\end{example}

\begin{example} \label{ex8} The next building block is
given by the algebraic surface of degree $6$ with implicit
equation
$$(x^2 + y^2 + z^2) z^4 - 2 (8 x^2 + 9 y^2 + 9 z^2) x z^2 - 27 (y^2 + z^2)^2=0.$$
In parametric form it can be written as:
\begin{multline}\label{eq8}
\mathbf{r}_8(u,v)=\frac{1}{1+u^2+v^2}\left(\,
u^4-3 u^2v^2-3u^2, \, 4 u^3 v,\, 4 u^3\,\right).\\[-7pt]
\end{multline}
Its cyclographic preimage can be parametrized as
\begin{multline}
\mathbf{R}_8(\varphi,\lambda)
= \left(\, 0,\, 0 ,\, -\varphi^3,\, \varphi^3
\, \right)
+ \lambda \left(1 ,\, 0,\, -\varphi,\, \varphi\, \right) \\[-7pt]
\end{multline}
\end{example}

\begin{example} \label{ex9} Another example is the algebraic surface of degree $8$,
given implicitly by:
$$z^2(y^2+z^2)\left(z^2-4y-4\right)^2 +
  x^2 \left(64 y^3 - 24 (y-3) y z^2 - 6 (y+6) z^4 + z^6\right)-27 x^4 z^2 =0.$$
In parametric form:
\begin{multline}\label{eq9}
\mathbf{r}_9(u,v)=\frac{1}{1+u^2+v^2}\left(\,
2 uv(u^2-v^2-1), \, u^2 (3v^2-u^2-1), \,4 u^2 v
\,\right).\\[-7pt]
\end{multline}
Its cyclographic preimage can be parametrized as
\begin{multline}
\mathbf{R}_9(\varphi,\lambda)
= \left(\, 0,\, -\varphi^2 ,\, 0,\, 0
\, \right)
+ \lambda \left(1 ,\, 0,\, -\varphi,\, \varphi\, \right) \\[-7pt]
\end{multline}
This surface has the following property: there are two $1$-parametric families of cones touching the surface along certain curves.
The other family can be parametrized as
\begin{multline}
\left(\, 0,\, 0 ,\, \varphi+\varphi^3,\, \varphi-\varphi^3
\, \right)
+ \lambda \left(0 ,\, 1,\, -\varphi,\, \varphi\, \right). \\[-7pt]
\end{multline}
\end{example}

Finally, we have the following two ``monsters''. We do not write their implicit equations because this would take several pages.

\begin{example} \label{ex10}First the algebraic surface of degree not greater than $14$
described by
\begin{multline}\label{eq10}
\mathbf{r}_{10}(u,v)=\frac{1}{1+u^2+v^2}
\left(\,
\text{\begin{tabular}{c}
$u^5 - 2 u^3 (1 + 4 v^2) + 3 u (v^2 + v^4)$\\
$3 u^2 v (1 + 2 u^2 -   2 v^2)$\\
$3 u^2(u^2 - 3 v^2)$
\end{tabular}}
\,\right).\\[-7pt]
\end{multline}
Its cyclographic preimage can be parametrized as
\begin{multline}
\mathbf{R}_{10}(\varphi,\lambda)
= \left(\, 0,\, 0 ,\, -3\varphi^2 - 4\varphi^4,\, -3\varphi^2 + 4\varphi^4
\, \right)\!/2 \,+\, \lambda \left(1 ,\, 0,\, -\varphi,\, \varphi\, \right). \\[-7pt]
\end{multline}
\end{example}

\begin{example} \label{ex11} The second monster is the
algebraic surface of degree not greater than $18$ with
parametrization
\begin{multline}
\mathbf{r}_{11}(u,v)=\frac{1}{1+u^2+v^2}
\left(\,
\text{\begin{tabular}{c}
$3 u^6 - 5 u^4 (1 + 6 v^2) + 15 u^2 (v^2 + v^4)$\\
$2 u^3 v (5 +   9 u^2 - 15 v^2)$\\
$8 u^3 (u^2 - 5  v^2)$
\end{tabular}}
\,\right).\\[-7pt]
\end{multline}
Its cyclographic preimage can be parametrized as
\begin{multline}
\mathbf{R}_{11}(\varphi,\lambda)
= \left(\, 0,\, 0 ,\, -5\varphi^3 - 6\varphi^5,\, -5\varphi^3 + 6\varphi^5
\, \right)
+ \lambda \left(1 ,\, 0,\, -\varphi,\, \varphi\, \right). \\[-7pt]
\end{multline}

\end{example}

An L-minimal surface enveloped by a parabolic family of cones can be obtained from Examples~\ref{ex7}--\ref{ex11} by performing rotations and taking special convolution surfaces; see Figure~\ref{figure-parabolic-general}:


\begin{corollary}[Classification for parabolic type]\label{classification-parabolic}
A Laguerre minimal surface
enveloped by a parabolic family of cones 
is Laguerre equivalent to a piece of the surface
\begin{multline} \label{eq-classification-parabolic}
\mathbf{r}(u,v)=a_1\mathbf{r}^{\theta}_7(u,v)+a_2\mathbf{r}_8(u,v)
+a_3\mathbf{r}_9(u,v)+a_4\mathbf{r}^{\pi/2}_9(u,v)+a_5\mathbf{r}_{10}(u,v)+a_6\mathbf{r}_{11}(u,v)\\[-10pt]
\end{multline}
for some $a_1,a_2,a_3,a_4,a_5,a_6,\theta\in\mathbb{R}$.
Conversely, an immersed piece of surface~\textup{(\ref{eq-classification-parabolic})} is Laguerre minimal and is enveloped by a parabolic family of cones.
\end{corollary}

\begin{figure}[htbp]
\centering
\includegraphics[width=\textwidth]{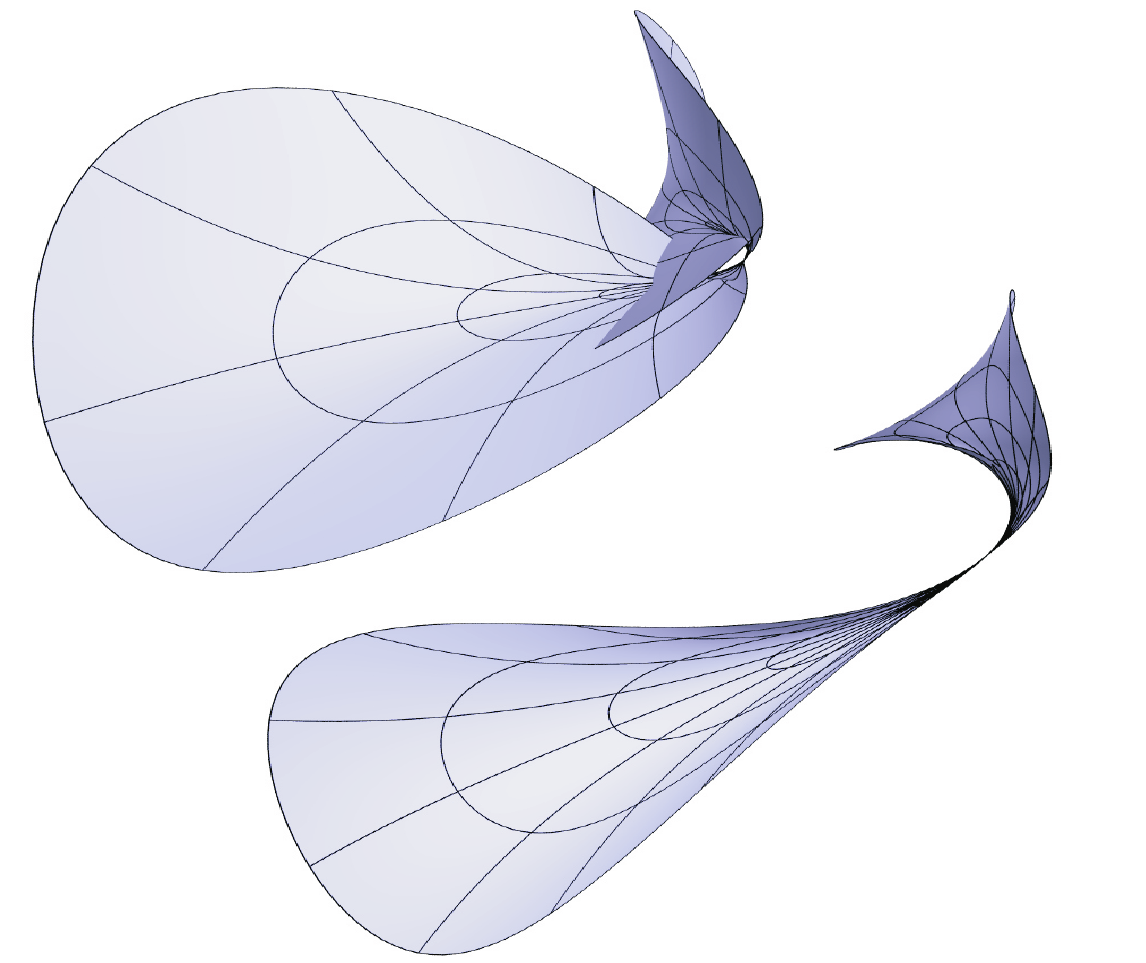}
 \caption{Two general L-minimal surfaces enveloped by parabolic families of cones. For details
 refer to Definition~\ref{def-parabolic} and Corollary~\ref{classification-parabolic}.}
 \label{figure-parabolic-general}
\end{figure}

\begin{proof}[of Corollary~\ref{classification-parabolic}]
Let us prove the direct implication.
Let $\Phi$ be an L-minimal surface enveloped by a parabolic family of cones.
Then the surface $\Phi^i$ carries a family of i-M-circles such that the top view of the family is a parabolic pencil. By Theorem~\ref{i-classification-parabolic}
it follows that the surface $\Phi^i$ is i-M-equivalent to surface~(\ref{eq-th2-1}).

Right-hand side of formula~(\ref{eq-th2-1}) is a linear combination in the expressions in the left column of Table~\ref{trans-parabolic}.
Performing an i-M-transformation $z\mapsto z+a(x^2+y^2)+bx+cy+d$ one can eliminate the last expression from the linear combination.
By Propositions~\ref{cor1} and Table~\ref{isotropic-trans} transformation~(\ref{eq-cor1}) takes the functions in the left column of Table~\ref{trans-parabolic} to the surfaces in the right column. Since the transformation~(\ref{eq-cor1}) is linear in $F$ the direct implication follows.

The converse implication follows from Propositions~\ref{cor1},~\ref{observation}, Theorem~\ref{pgm}, and Table~\ref{trans-parabolic}.
\qed\end{proof}

Finally we describe the cyclographic preimage (the proof is analogous to the proof of Theorem~\ref{cyclo-elliptic}).

\begin{theorem}
	The cyclographic preimage of a Laguerre minimal surface enveloped by
	a parabolic family of cones is up to Lorentz transformations a piece of the surface
	\begin{multline}
	\mathbf{R}(\varphi,\lambda)=
	\left(\text{\begin{tabular}{c}
	$0$\\
	$A\varphi+B\varphi^2$\\
        $C\varphi+D\varphi^2+E\varphi^3+F(3\varphi^2+4\varphi^4)+G(5\varphi^3+6\varphi^5)$\\
	 $C\varphi-D\varphi^2-E\varphi^3+F(3\varphi^2-4\varphi^4)+G(5\varphi^3-6\varphi^5)$
	\end{tabular}}\right)
	+
	\lambda\left(\text{\begin{tabular}{c}
	$1$\\
	$0$\\
	$-\varphi$\\
	$\varphi$
	\end{tabular}}\right)
	\\[-7pt]
	\end{multline}
	for some $A,B,C,D,E,F,G\in\mathbb{R}$.
\end{theorem}








\subsection{Open problems}

\begin{conjecture} A surface such that there is a $2$-parametric family of cones of revolution touching the surface along certain curves distinct from directrices is either a sphere or a parabolic cyclide.
\end{conjecture}

\begin{problem} Describe all surfaces such that there are two $1$-parametric families of cones of revolution touching the surface along curves.
\end{problem}

\begin{problem} Describe all Willmore surfaces such that there is a $1$-parametric family of circles lying in the surface.
\end{problem}

\section*{Acknowledgements}
The authors are grateful to S.~Ivanov for useful discussions.
M.~Skopenkov was supported in part by M\"obius Contest Foundation for Young Scientists and the Euler Foundation. H.~Pottmann and P.~Grohs are partly supported by the Austrian Science Fund (FWF) under grant S92.

\bibliographystyle{amsplain}
\bibliography{laguerre}

\end{document}